\documentclass{article}

\usepackage[T2A]{fontenc}
\usepackage[utf8]{inputenc}
\usepackage[russian,english]{babel}
\usepackage{amsfonts}
\usepackage{amsthm}
\usepackage{amsmath}
\usepackage{amssymb}
\usepackage{cancel}
\usepackage{tabu}
\usepackage{array}
\usepackage{tikz}
\usepackage{makecell}
\usepackage{xcolor}
\usepackage{graphicx}
\usepackage[final]{pdfpages}
\usepackage{bbm}

\usepackage{titlesec}
\usepackage{indentfirst}
\usepackage{enumitem}
\usepackage{caption}
\usepackage{refcount}
\usepackage{cite}
\usepackage[colorlinks=true, linkcolor=blue, citecolor=blue, urlcolor=blue, hypertexnames=false]{hyperref}
\usepackage{titling}
\usepackage{etoc}

\titleformat{\section}[block]          
  {\centering\large\bfseries}          
  {\S\,\thesection.}                   
  {1em}                                
  {}                                   

\makeatletter
\renewenvironment{proof}[1][\proofname]{\par
  \pushQED{\qed}%
  \normalfont \topsep6\p@\@plus6\p@\relax
  \trivlist
  \item[\hspace*{\parindent}\hskip\labelsep \bfseries #1\@addpunct{.}]%
  \ignorespaces
}{%
  \popQED\endtrivlist\@endpefalse
}
\makeatother

\graphicspath{ {./images/} }

\renewcommand{\ge}{\geqslant}

\renewcommand{\le}{\leqslant}

\newtheoremstyle{style2}
  {\topsep}
  {\topsep}
  {\itshape}
  {\parindent}
  {\bfseries}
  {.}
  {5pt plus 1pt minus 1pt}
  {}

\title{Граница Мартина графа прыжков по подграфам графа Юнга --- Фибоначчи}

\author{В.~Ю.~Евтушевский}

\begin{document}

\selectlanguage{russian}

\maketitle

\etocsettagdepth{ru}{all}
\etocsettagdepth{en}{none}
\tableofcontents

\theoremstyle{style2}
\newtheorem{Lemma}{Лемма}

\theoremstyle{style2}
\newtheorem{Alg}{Алгоритм}

\theoremstyle{style2}
\newtheorem{Col}{Следствие}

\theoremstyle{style2}
\newtheorem{theorem}{Теорема}

\theoremstyle{style2}
\newtheorem{Def}{Определение}

\theoremstyle{style2}
\newtheorem{Prop}{Предложение}

\theoremstyle{style2}
\newtheorem{Problem}{Задача}

\theoremstyle{style2}
\newtheorem{Zam}{Замечание}

\theoremstyle{style2}
\newtheorem{Oboz}{Обозначение}

\theoremstyle{style2}
\newtheorem{Ex}{Пример}

\theoremstyle{style2}
\newtheorem{Nab}{Наблюдение}

\begin{abstract}
Для графа прыжков по подграфу графа Юнга --- Фибоначчи, образованного словами с не более чем $K$ двойками, получены формулы для числа путей между вершинами и описана граница Мартина пространства путей.
\end{abstract}

\etocdepthtag.toc{ru}
\section{Введение}

Пусть $G=(V,E)$ --- ациклический ориентированный граф. Ему соответствует частично упорядоченное множество $(V,\le)$ с отношением достижимости: $v\le w$, если есть путь из $v$ в $w$. Рассмотрим градуированный граф, на нулевом уровне которого находится корень, а $i$-й уровень при $i>0$ есть $V\times \{i\},$ и ребро из $(v,i)$ в $(w,i+1)$ ведёт, если $v\leqslant w$. Его естественно называть графом прыжков по $G$. Естественные вопросы о градуированном графе --- перечисление путей между вершинами (в данном случае --- перечисление нестрогих цепей данной длины между двумя элементами частично упорядоченного множества $(V,\le)$) и описание вероятностных центральных мер на пространстве путей из корня в бесконечность (центральность означает, что все начальные отрезки случайного пути с фиксированной вершиной $n$-го уровня равновероятны). Для градуированных графов Юнга и Юнга --- Фибоначчи эти задачи давно решены. Для графов прыжков дело обстоит сложнее, хотя в случае графа Юнга и можно написать детерминантную формулу для числа путей между вершинами. Настоящая работа посвящена нахождению границы Мартина графа прыжков по естественному подграфу графа Юнга --- Фибоначчи, образованному словами с ограниченным количеством двоек (см. подробные определения далее).

\begin{figure}[htbp]
    \centering
    \includegraphics[width=12cm, height=10cm]{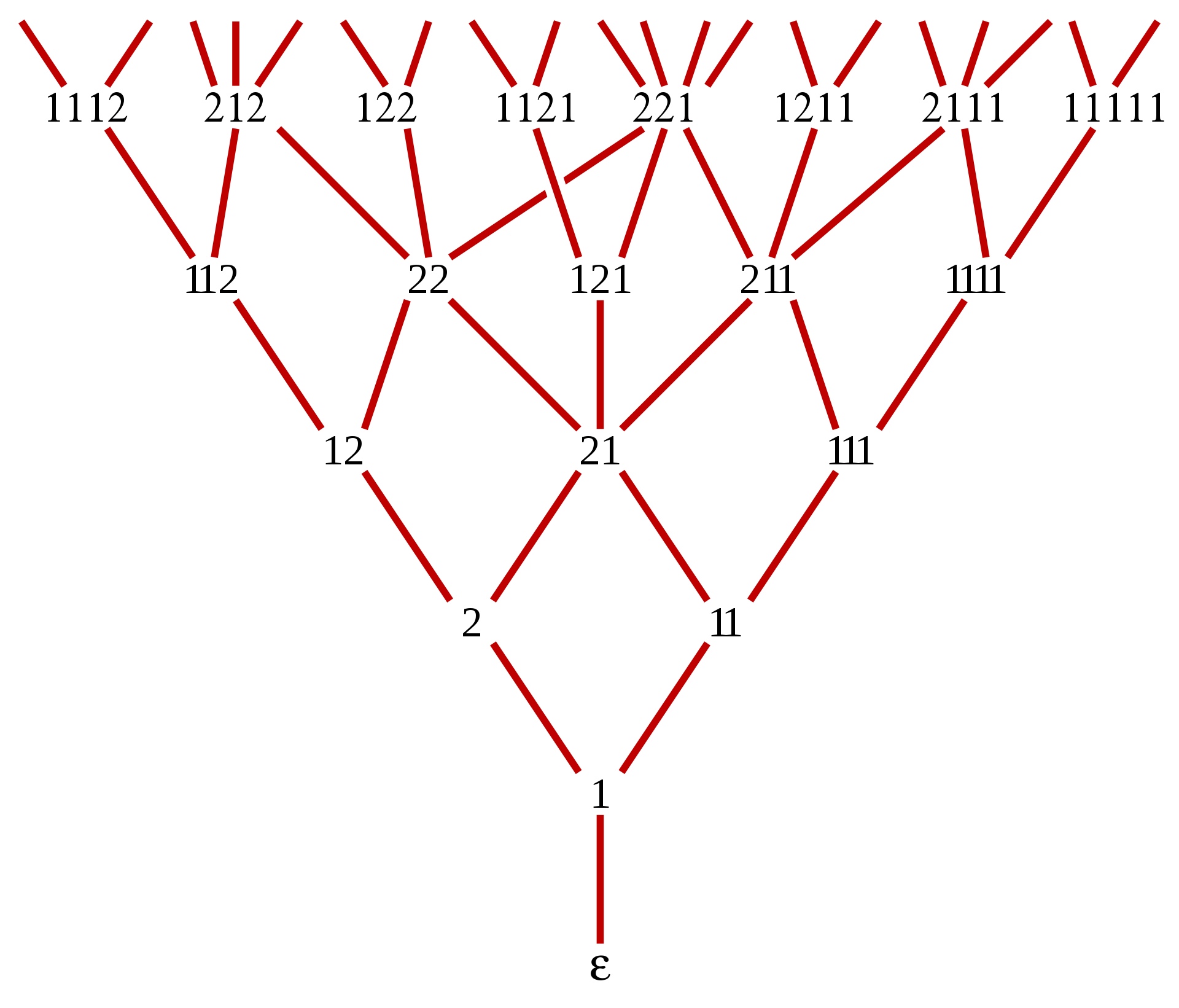}
    \caption{Граф Юнга --- Фибоначчи}
    \label{yf}
\end{figure}

Рассмотрим слова над алфавитом $\{1,2\}$ с данной суммой цифр $n.$ Как известно, их количество есть число Фибоначчи $F_{n+1}$ ($F_0=0,$ $F_1=1,$ $F_{k+2}=F_{k+1}+F_k$), и это самая распространённая комбинаторная интерпретация чисел Фибоначчи. Также можно думать о разбиениях полосы $2\times n$ на домино $1\times 2$ и $2\times 1,$ сопоставляя двойки парам горизонтальных домино, а единицы вертикальным домино.

Введём на этом множестве слов частичный порядок: будем говорить, что слово $v$ \emph{предшествует} слову $w$, если после удаления самого длинного общего суффикса в слове $w$ остаётся не меньше двоек, чем в слове $v$ остаётся цифр. 

Это действительно частичный порядок, более того, соответствующее частично упорядоченное множество является модулярной решёткой, известной как решётка Юнга --- Фибоначчи.

Графом Юнга --- Фибоначчи (он изображён на рис. \ref{yf}) называют диаграмму Хассе этой решётки. Это градуированный граф, который мы представляем растущим снизу вверх, начиная с пустого слова. Градуировкой служит функция суммы цифр. Опишем явно, как устроены ориентированные рёбра. Рёбра <<вверх>> из данного слова $v$ ведут в слова, получаемые из $v$ одной из двух операций:
\begin{enumerate}
    \item заменить самую левую единицу на двойку;
    \item вставить единицу в такое место,
    что левее неё нет единиц.
\end{enumerate}

В данной работе будем считать граф Юнга --- Фибоначчи ориентированным и содержащим исключительно рёбра <<вверх>>.

Этот граф помимо модулярности является $1$-дифференциальным, то есть для каждой вершины исходящая степень на $1$ превосходит входящую степень.

Изучение градуированного графа Юнга --- Фибоначчи было инициировано в 1988 г. одновременно и независимо друг от друга такими математиками, как Ричард Стенли \cite{St} и Сергей Владимирович Фомин \cite{Fo}.

Одна из причин интереса к нему в том, что существует всего две корневые локально конечные $1$-дифференциальные модулярные решётки, вторая --- это решётка диаграмм Юнга, имеющая ключевое значение в теории представлений симметрической группы \cite[Theorem~1.3]{Bur}. 

Ключевые вопросы о градуированных графах касаются центральных мер на пространстве (бесконечных) путей в графе. Эта точка зрения последовательно развивалась в работах Анатолия Моисеевича Вершика, к недавнему обзору которого \cite{Ver} и приводимой там литературе мы отсылаем читателя.

Среди центральных мер выделяют те, которые являются пределами мер, индуцированных путями в далёкие вершины --- так называемую \emph{границу Мартина} графа.

Граница пространства путей графа Юнга --- Фибоначчи изучалась в работе Фредерика Гудмана и Сергея Васильевича Керова (2000) \cite{GK}. Они использовали алгебраический формализм Окады \cite{Okada}.

Как следует из самого определения, асимптотический вопрос о границе напрямую связан с перечислительным вопросом о числе путей между двумя вершинами графа. Отметим важную общую работу Фомина \cite{Fo1} о перечислении путей в градуированных графах, в которой приводится ряд общих тождеств и указывается связь, помимо прочего, с обобщением алгоритма Робинсона --- Шенстеда --- Кнута.

В работе автора \cite{Evtuh1} (краткая версия которой была опубликована как \cite{Evtuh2}) была предложена явная формула для числа путей между двумя вершинами в графе Юнга --- Фибоначчи, а в работе Ивана Алексеевича Бочкова и автора \cite{BE} для всех мер границы Мартина графа Юнга --- Фибоначчи была найдена явная формула для меры множества всех бесконечных путей из корня, проходящих через данную вершину.

В разделах 2 и 3 мы предлагаем формулы для числа путей между двумя вершинами в графе прыжков по графу Юнга --- Фибоначчи. В разделе 5 мы рассматриваем подграф графа Юнга --- Фибоначчи, образованный словами с не более чем $K$ двойками при произвольном целом неотрицательном $K.$ Его можно рассматривать как аналог подграфа графа Юнга, образованного диаграммами с не более чем $K$ строками. Мы полностью описываем границу Мартина графа прыжков по этому графу, и для каждой меры этой границы приводим формулу для меры множества всех бесконечных путей из корня, проходящих через данную вершину.

\section{Основные понятия и обозначения}

Через $\mathbb{N}_0$ и $\mathbb{N}$ обозначаются множества целых неотрицательных и целых положительных (натуральных) чисел, соответственно. Для $n\in \mathbb{N}_0$ положим $\overline{n}:=\{0,1,\ldots,n\}.$

Пусть $\mathbb{YF}$ --- это граф Юнга --- Фибоначчи, его вершины --- конечные слова из $1$ и $2.$ Далее $\mathbb{YF}$ как множество --- это множество именно вершин графа Юнга --- Фибоначчи. Для $v,w\in\mathbb{YF}$ определено слово $vw$ --- конкатенация $v$ и $w.$ Также для $v\in\mathbb{YF}$ и $n\in\mathbb{N}_0$ естественным образом определяется слово $v^n$.

Для $v \in \mathbb{YF}$ сумму цифр слова $v$ обозначим через $|v|$. Число $|v|$ --- \emph{ранг} вершины $v$ в графе $\mathbb{YF}$ (считаем, что пустое слово имеет ранг $0$ и переход по ребру наверх увеличивает ранг на $1$). В данной работе будем считать граф $\mathbb{YF}$ ориентированным с рёбрами из вершин меньшего ранга в вершины большего ранга.

Для $v,w\in \mathbb{YF}$ количество путей из $v$ в $w$ обозначим через $d(v,w).$ В частности, $d(v,w)=0$ при $|w|<|v|.$

Количество цифр в $v\in\mathbb{YF}$ обозначим через $\#v,$ а количество двоек в $v$ --- через $d(v).$ Очевидно, что $\#v+d(v)=|v|.$ 

Также пусть $v,w \in \mathbb{YF}.$ У вершин $v,w$ есть самый длинный общий суффикс $u\in \mathbb{YF},$ то есть самое длинное слово, для которого $v=v'u,$ $w=w'u.$ Положим $w_v:=w',$ $v_w:=v',$ $r(v,w):=d(w_v)-\# v_w$. Из определения графа Юнга --- Фибоначчи следует, что  
\begin{align*}
    r(v,w)\ge 0&\Longleftrightarrow d(v,w)\ge 1;\\
    r(v,w)= 0&\Longrightarrow d(v,w)=1.
\end{align*}
В случае $r(v,w)\ge 0$ будем писать $v\le w,$ в случае $r(v,w)= 0$ будем писать $v\lesssim w,$ а в случае $r(v,w)> 0$ будем писать $v\lnsim w.$   

Итак, определим ориентированный градуированный граф $\overline{\mathbb{YF}}$ (\emph{граф прыжков по графу Юнга --- Фибоначчи}). Его множеством вершин является множество $\overline{\mathbb{YF}}:=\{(\varepsilon,0)\}\cup \mathbb{YF}\times \mathbb{N},$ причём на уровне $0$ находится ровно одна вершина (корень) $(\varepsilon,0),$ а при каждом фиксированном $\ell\in\mathbb{N}$ на уровне $\ell$ находятся все вершины вида $(v,\ell),$ $v\in\mathbb{YF}.$ Ориентированное ребро между вершинами $(v,\ell),(w,\ell+1)\in \overline{\mathbb{YF}}$ проведено, если и только если $v\le w.$ 

Если $(v,\ell), (w,\ell+n)\in \overline{\mathbb{YF}},$ то количество путей из $(v,\ell)$ в $(w,\ell+n)$ обозначим через $D(v,w,n).$ В частности, $D(v,w,n)=0$ при $n<0$ или при $v\nleqslant w.$ Также очевидны следующие соотношения:
\begin{align}
   \text{При } n\in\mathbb{N}_0 \qquad D(v,\varepsilon,n)=&\mathbbm{1}_{v=\varepsilon};\notag \\
   D(v,w,0)=&\mathbbm{1}_{v=w}; \notag \\
   \text{При } n\in\mathbb{N} \qquad D(v,w,n)=&\sum_{u\le w}D(v,u,n-1);\notag\\
   \text{При } v\lesssim w,\; n\in\mathbb{N}_0\qquad   D(v,w,n) =& \binom{|w|-|v|+n-1}{|w|-|v|}. \label{r=0}
\end{align}

\section{Первая формула для $D(v,w,n)$}

Из определения графа Юнга --- Фибоначчи следует, что для любой вершины $w\in\mathbb{YF}$
\begin{gather*}
\{u:\; u\le 1w\}= \{u:\; u\le w\}\cup \{1w\}; \qquad \{u:\; u\le 2w\} = \{\varepsilon\} \cup\{1u:\; u\le w\}\cup \{2u:\; u\le w\} .
\end{gather*}
Таким образом, очевидно, что для любых $v,w\in\mathbb{YF},$ $n\in\mathbb{N}$
\begin{align}
D(v,1w,n)=&D(v,w,n)+D(v,1w,n-1);\label{e1}\\
D(v,2w,n)=&\mathbbm{1}_{v=\varepsilon}+\sum_{u\le w}(D(v,1u,n-1)+D(v,2u,n-1)).\label{e2}
\end{align}

\begin{theorem}\label{rf}
Пусть $v,w\in {\mathbb{YF}},$ $n\in\mathbb{N}_0.$ Тогда
\begin{align}
        D(v,1w,n)=& \mathbbm{1}_{v=1w}+ \sum_{m=1}^n D(v,w,m);\label{rf1}\\
        D(\varepsilon,2w,n)=&\sum_{m=1}^n m\cdot D(\varepsilon,w,m);\label{rf2}\\
        D(1v,2w,n)=&\sum_{m=1}^n D(v,w,m)+\sum_{m=1}^n (m-1)\cdot D(1v,w,m)\label{rf3};\\
        D(2v,2w,n)=&D(v,w,n)+\sum_{m=1}^n (m-1)\cdot D(2v,w,m).\label{rf4}
\end{align}
\end{theorem}
\begin{proof}
    Докажем теорему по индукции по $n.$ База $n=0$ для всех равенств очевидна. Переход к $n$ для равенства \eqref{rf1} очевиден из формулы \eqref{e1}. Далее, последовательно докажем переход к $n$ для равенств \eqref{rf2},\eqref{rf3},\eqref{rf4}. В каждом случае воспользуемся формулой \eqref{e2}, уже доказанным равенством \eqref{rf1} и предположением индукции:
   \begin{align*}
    D(\varepsilon,2w,n)=&1+\sum_{u\le w}(D(\varepsilon,1u,n-1)+D(\varepsilon,2u,n-1))=\\
    =&1+\sum_{u\le w}\left(\sum_{m=1}^{n-1} D(\varepsilon,u,m)+\sum_{m=1}^{n-1} m\cdot D(\varepsilon,u,m)  \right)\\=&1+\sum_{m=2}^{n} m\cdot D(\varepsilon,w,m)=\sum_{m=1}^{n} m\cdot D(\varepsilon,w,m);\\
    D(1v,2w,n)=&\sum_{u\le w}(D(1v,1u,n-1)+D(1v,2u,n-1))\\
    =&\sum_{u\le w}\left(\mathbbm{1}_{v=u} +\sum_{m=1}^{n-1} D(1v,u,m)+\sum_{m=1}^{n-1}  D(v,u,m)+\sum_{m=1}^{n-1} (m-1)\cdot D(1v,u,m)  \right)\\
    =&D(v,w,1)+\sum_{m=2}^{n}  D(v,w,m)+\sum_{m=2}^{n} (m-1)\cdot D(1v,w,m)\\
    =&\sum_{m=1}^n D(v,w,m)+\sum_{m=1}^n (m-1)\cdot D(1v,w,m);\\
    D(2v,2w,n)=&\sum_{u\le w}(D(2v,1u,n-1)+D(2v,2u,n-1))\\
    =&\sum_{u\le w}\left(\sum_{m=1}^{n-1} D(2v,u,m)+  D(v,u,n-1)+\sum_{m=1}^{n-1} (m-1)\cdot D(2v,u,m)  \right)\\
    =&D(v,w,n)+\sum_{m=1}^n (m-1)\cdot D(2v,w,m).\\
    \end{align*}
\end{proof}
\section{Вторая формула для $D(v,w,n)$}
Для таких $a,b,n\in\mathbb{N}_0,$ что $a\le b,$ имеют место следующие тождества:
\begin{align}
    \sum_{m=1}^n \binom{a+m-1}{b}=&\binom{a+n}{b+1};\label{f4}\\
    \sum_{m=1}^n m\binom{a+m-1}{b}=&(b+1)\binom{a+n}{b+2}+(b+1-a)\binom{a+n}{b+1};\label{f5}\\
    \sum_{m=1}^n (m-1)\binom{a+m-1}{b}=&(b+1)\binom{a+n}{b+2}+(b-a)\binom{a+n}{b+1}.\label{f6}
\end{align}    
Они проверяются по индукции по $n$ и будут использоваться в дальнейшем.

\renewcommand{\labelenumi}{\arabic{enumi}$^\circ$}

Также потребуются новые обозначения. При $w,v\in\mathbb{YF}$ положим $s(v,w):=\#(w_v)-d(v_w).$ Сразу заметим, что $r(v,w)+s(v,w)=|w|-|v|,$ а также то, что при $v\le w$ выполняется $s(v,w)\ge 0.$

Следующая функция $F$ позволяет найти новую формулу для $D(v,w,n).$

Функция 
\begin{gather*}
  F: \quad \{(v,w,i)\in \mathbb{YF} \times \mathbb{YF}\times \mathbb{Z}:\; v\le w\} \to \mathbb{Z}  
\end{gather*}
определяется рекурсивно следующим образом:

\begin{itemize}
    \item Если $v\lesssim w,$ то
    \begin{gather*}
      F(v,w,i)=\mathbbm{1}_{i=0};  
    \end{gather*}
    \item Если $v\lnsim 1w,$ то 
    \begin{gather*}
        F(v,1w,i)=F(v,w,i);
    \end{gather*}
    \item    
    \begin{gather*}
    F(\varepsilon,2w,i)=(|w|+1-i) \cdot F(\varepsilon,w,i) + (d(w)+2-i) \cdot F(\varepsilon,w,i-1);    
    \end{gather*} 
    \item Если $v\lnsim 2w,$ $v=\alpha v',$ $\alpha\in\{1,2\},$ то
        \begin{align*}
            F(v,2w,i)=&F(v',w,i)+\mathbbm{1}_{v,w}\cdot F(v',w,i-1)+(|w|-|v|+1-i) \cdot F(v,w,i) \\&+ (r(v,w)+1-i) \cdot F(v,w,i-1),
        \end{align*}
        где
    \begin{gather*}
    \mathbbm{1}_{v,w}=
    \begin{cases}
    1,\quad \text{если } \alpha=1,\;v_w=\varepsilon,\\
    0,\quad \text{иначе.}
    \end{cases}
\end{gather*}
\end{itemize}

\begin{Prop}\label{z1}
При данных $v\le w$ значения функции $F(v,w,i)$ являются натуральными при $i\in \overline{r(v,w)}$ и равны нулю иначе. 
\end{Prop}
\begin{proof}
    Данное предложение несложно доказать по индукции, используя то, что если $v\le w,$ то
\begin{align*}
    r(v,2w)=&r(v,w)+1;\\
    r(v,2w)=&r(v',w)+\mathbbm{1}_{v,w}.
\end{align*}    
\end{proof}

\begin{theorem}\label{df}
Пусть $n\in\mathbb{N}_0,$ $v,w\in\mathbb{YF}:$ $v\le w.$ Тогда
\begin{gather*}
  D(v,w,n)=\sum_{i=0}^{r(v,w)}F(v,w,i)\binom{s(v,w)+n-1}{|w|-|v|- i}. 
\end{gather*}
\end{theorem}
\begin{proof}
Будем доказывать теорему индукцией по $r(v,w),$ а при данном значении $r(v,w)$ --- индукцией по $|w|.$ 

В доказательстве будем использовать то, что в силу предложения \ref{z1} при любых $n\in\mathbb{N}_0,$ $v,w\in\mathbb{YF}:$ $v\le w$ выполняется
\begin{gather*}
    \sum_{i=0}^{r(v,w)}F(v,w,i)\binom{s(v,w)+n-1}{|w|-|v|- i }=\sum_{i=-\infty}^{\infty}F(v,w,i)\binom{s(v,w)+n-1}{|w|-|v|- i}.\label{infinity}
\end{gather*}

База $v\lesssim w$ сводится к \eqref{r=0}. Докажем переход к $(v,\gamma w)$ при $\gamma\in\{1,2\},$ $v\lnsim\gamma w.$

В случае $\gamma=1$ воспользуемся \eqref{rf1}, предположением индукции и \eqref{f4}:
\begin{align*}
    D(v,1w,n)=\sum_{m=1}^{n} D(v,w,m)=&\sum_{i=0}^{r(v,w)}F(v,w,i)\sum_{m=1}^n\binom{s(v,w)+m-1}{|w|-|v|-i}\\
    =&\sum_{i=-\infty}^{\infty}F(v,1w,i)\binom{s(v,1w) +n-1}{|1w|-|v|-i}.
\end{align*}

В случае $v=\varepsilon$ воспользуемся \eqref{rf2}, предположением индукции и \eqref{f5}:
\begin{align*}
    D(\varepsilon,2w,n)=&\sum_{m=1}^n m\cdot D(\varepsilon,w,m)=\sum_{i=0}^{d(w)}F(\varepsilon,w,i) \sum_{m=1}^n m \binom{\#w+m-1}{|w|-i}\\
    =&\sum_{i=-\infty}^{\infty}F(\varepsilon,w,i) \left( (|w|+1-i) \binom{\#(2w)+n-1}{|2w|-i} + (d(w)+1-i) \binom{\#(2w)+n-1}{|2w|-1-i} \right)\\
    =&\sum_{i=-\infty}^{\infty}F(\varepsilon,2w,i)\binom{\#(2w)+n-1}{|2w|-i}.   
\end{align*}

Более содержательным является случай $\gamma=2,$ $v\ne \varepsilon.$ Заметим, что по предположению индукции и \eqref{f6}
\begin{multline}
    \sum_{m=1}^n  (m-1)\cdot D(v,w,m)=\sum_{i=0}^{r(v,w)} F(v,w,i)\sum_{m=1}^n (m-1)\binom{s(v,w) +m-1}{|w|-|v|-i}\\
    =\sum_{i=-\infty}^{\infty}F(v,w,i) \biggl( (|w|-|v|+1-i) \binom{s(v,2w)+n-1}{|2w|-|v|-i}+(r(v,w)-i) \binom{s(v,2w)+n-1}{|2w|-|v|-1-i} \biggr).\label{e3}
\end{multline}

Далее рассмотрим два случая:

\renewcommand{\labelenumi}{\arabic{enumi}$^\circ$}

\begin{enumerate}
    \item Пусть $v=1v'.$ Заметим, что в силу предположения индукции, \eqref{f4} и того факта, что в данном случае $s(v',w)=s(v,2w)-1+\mathbbm{1}_{v,w},$ выполняется
    
    \begin{align*}
    \sum_{m=1}^n  D(v',w,m)=&\sum_{i=0}^{r(v',w)}F({v'},w,i)\sum_{m=1}^n\binom{s(v',w)+m-1}{|w|-|v'|- i }\\
    =&\sum_{i=-\infty}^{\infty}F({v'},w,i)\binom{s(v,2w)+n-1+\mathbbm{1}_{v,w}}{|2w|-|v|- i }\\
    =&\sum_{i=-\infty}^{\infty}F({v'},w,i)\binom{s(v,2w)+n -1}{|2w|-|v|- i}
    +\mathbbm{1}_{v,w}\cdot\sum_{i=-\infty}^{\infty}F({v'},w,i)\binom{s(v,2w)+n -1}{|2w|-|v|-1- i}.
    \end{align*}
    Сложим данное равенство с равенством \eqref{e3}, применим \eqref{rf3} и получим требуемое.
    \item Пусть $v=2v'.$ Заметим, что в силу предположения индукции и того факта, что в данном случае $s(v',w)=s(v,2w),$ выполняется
\begin{align*}
    D(v',w,n)=\sum_{i=-\infty}^{\infty}F(v',w,i)\binom{s(v,2w)+n-1}{|2w|-|v|-i}.
\end{align*}
    Сложим данное равенство с равенством \eqref{e3}, применим \eqref{rf4} и получим требуемое.
\end{enumerate}

\end{proof}

\section{Многочлены $Q_{v,w}(p)$}

Следующие многочлены $Q_{v,w}(p)$ тесно связаны с функцией $F(v,w,n)$ и будут использоваться в дальнейшем.

При данных $v,w\in\mathbb{YF}:$ $v\le w$ многочлен $Q_{v,w}(p)$ определяется рекурсивно следующим образом:
\begin{itemize}
    \item Если $v\lesssim w,$ то
    \begin{gather*}
       Q_{v,w}(p)=1; 
    \end{gather*}
    \item Если $v\lnsim 1w,$ то
    \begin{gather*}
      Q_{v,1w}(p)=Q_{v,w}(p);  
    \end{gather*}
    \item 
    \begin{gather*}
        Q_{\varepsilon,2w}(p)=(1-p)\cdot(p\cdot Q_{\varepsilon,w}(p))'+(|w|+1)\cdot p \cdot Q_{\varepsilon,w}(p);
    \end{gather*}
    \item Если $v\lnsim 2w,$ $v=\alpha v',$ $\alpha\in\{1,2\},$ то
    \begin{gather*}
      Q_{v,2w}(p)=p\cdot(1-p)\cdot Q_{v,w}'(p)+(|w|-|v|+1)\cdot p \cdot Q_{v,w}(p)+Q_{v',w}(p).  
    \end{gather*}
\end{itemize}

\begin{Lemma}\label{qf}
Пусть $v,w\in\mathbb{YF}:$ $v\le w$. Тогда
\begin{gather*}
    Q_{v,w}(p)=\sum_{i=0}^{r(v,w)}F(v,w,i)\cdot(1-p)^i\cdot p^{r(v,w)-i}.
\end{gather*}
В частности, для любого $p\in(0,1)$ выполняется $Q_{v,w}(p)>0.$
\end{Lemma}
\begin{proof}
В силу предложения \ref{z1} данная лемма равносильна тому, что при $v,w\in\mathbb{YF}:$ $v\le w$ 
\begin{gather}
    Q_{v,w}(p)=\sum_{i=-\infty}^{\infty}F(v,w,i)\cdot(1-p)^i\cdot p^{r(v,w)-i}.\label{qinf}
\end{gather}

Будем доказывать \eqref{qinf} индукцией по $r(v,w),$ а при данном значении $r(v,w)$ --- индукцией по $|w|.$ База $v\lesssim w$ очевидна. Докажем переход к $(v,\gamma w)$ при $\gamma\in\{1,2\},$ $v\lnsim\gamma w.$ Сразу заметим, что случай $\gamma=1$ очевиден. Итак, пусть $\gamma=2.$ Заметим, что по предположению индукции
\begin{gather}
     (|w|-|v|+1)\cdot p\cdot Q_{v,w}(p)=(|w|-|v|+1)\cdot\sum_{i=-\infty}^{\infty}F(v,w,i)\cdot(1-p)^i\cdot p^{r(v,2w)-i}.\label{e5}
\end{gather}

Далее рассмотрим два случая:
\begin{enumerate}
    \item Пусть $v=\varepsilon.$ В данном случае по предположению индукции
\begin{multline}
        (1-p)\cdot (p\cdot Q_{\varepsilon,w}(p))'= (1-p)\cdot\left(\sum_{i=-\infty}^{\infty}F(\varepsilon,w,i)\cdot(1-p)^i\cdot p^{d(w)+1-i}\right)'\\
        =\sum_{i=-\infty}^{\infty}F(\varepsilon,w,i)(-i\cdot(1-p)^{i}\cdot p^{d(2w)-i}+(d(w)+1-i)\cdot(1-p)^{i+1}\cdot p^{d(2w)-1-i}).\label{eps}
\end{multline}
Сложим равенства \eqref{e5}, \eqref{eps} и получим требуемое.
    \item Пусть $v\ne \varepsilon.$ В данном случае по предположению индукции
    \begin{multline}
     p\cdot (1-p)\cdot Q'_{v,w}(p)=p\cdot (1-p)\cdot\left( \sum_{i=-\infty}^{\infty}F(v,w,i)\cdot(1-p)^i\cdot p^{r(v,w)-i}\right)'\\
     =\sum_{i=-\infty}^{\infty}F(v,w,i)(-i\cdot(1-p)^{i}\cdot p^{r(v,2w)-i}+(r(v,w)-i)\cdot(1-p)^{i+1}\cdot p^{r(v,2w)-1-i});\label{q1}
     \end{multline}
     \begin{multline}
     Q_{v',w}(p)=\sum_{i=-\infty}^{\infty}F(v',w,i)\cdot(1-p)^i\cdot p^{r(v,2w)-\mathbbm{1}_{v,w}-i}\\
     =\sum_{i=-\infty}^{\infty}F(v',w,i)\cdot(1-p)^i\cdot p^{r(v,2w)-i}+\mathbbm{1}_{v,w}\sum_{i=-\infty}^{\infty}F(v',w,i)\cdot(1-p)^{i+1}\cdot p^{r(v,2w)-1-i}.\label{q2}
    \end{multline}
    Сложим равенства \eqref{e5}, \eqref{q1}, \eqref{q2} и получим требуемое.
\end{enumerate}
\end{proof}

Далее нам понадобится следующее обозначение для любых $v,w\in\mathbb{YF}:$
\begin{equation*}
    \widetilde{Q}_{v,w}(p)=
    \begin{cases}
        Q_{v,w}(p), &\text{если } v\le w,\\
        1,&\text{если } \exists e\in\mathbb{N}:\; v = 1^ew,\\ 
        0,&\text{иначе.} 
    \end{cases}
\end{equation*}

\section{Описание границы Мартина}
Заметим, что если $a\in\mathbb{Z},$ $b\in\mathbb{N}_0,$ натуральный параметр $n$
стремится к бесконечности, то
\begin{align}
    \frac{(n+a)!}{n!}\sim& n^a;\label{sim1}\\
    \binom{n+a}{b}\sim&\frac{n^b}{b!}\label{sim2}.
\end{align}
В дальнейшем данные факты будут использоваться в доказательствах. 

Зафиксируем целое неотрицательное $K.$ 

Далее пусть ${}_K\mathbb{YF}$ --- это индуцированный подграф графа $\mathbb{YF}$ на вершинах вида $v\in\mathbb{YF},$ таких, что $d(v)\le K,$ а также ${}_K\overline{\mathbb{YF}}$ --- это индуцированный подграф графа $\overline{\mathbb{YF}}$ на вершинах вида $(v,n)\in\overline{\mathbb{YF}},$ таких, что $v\in{}_K{\mathbb{YF}}.$

Рассмотрим последовательность вершин $\tau'=\{(w'_m,n'_m)\}_{m=1}^\infty\in\left({}_K\overline{\mathbb{YF}}\right)^\infty,$ такую, что $n'_m\to\infty.$ Если для любой вершины $(v,\ell)\in{}_K\overline{\mathbb{YF}}$ существует предел
\begin{gather*}
  \mu_{\tau'}(v,\ell):=\lim_{m\to \infty} D(\varepsilon,v,\ell)\frac{D(v,w'_m,n'_m-\ell)}{D(\varepsilon,w'_m,n'_m)},  
\end{gather*}
то последовательность $\tau'$ порождает такую меру $\mu_{\tau'}$ на пространстве бесконечных путей в графе ${}_K\overline{\mathbb{YF}},$ что мера множества всех путей, проходящих через вершину $(v,\ell),$ равна $\mu_{\tau'}(v,\ell),$ при этом все $D(\varepsilon,v,\ell)$ начальных отрезков пути из корня в $(v,\ell)$ имеют равную меру $\mu_{\tau'}(v,\ell)/D(\varepsilon,v,\ell).$ Эти меры образуют так называемую \emph{границу Мартина} графа ${}_K\overline{\mathbb{YF}}.$ Итак, наша задача --- описать все такие меры. Кроме того, мера $\mu_{\tau'}$ называется \emph{вероятностной}, если для любого $\ell\in\mathbb{N}_0$ выполняется
\begin{gather*}
    \sum_{v\in{}_K\mathbb{YF}} \mu_{\tau'}(v,\ell)=1.
\end{gather*}
Мы также хотим узнать, какие меры границы Мартина графа ${}_K\overline{\mathbb{YF}}$ являются вероятностными. Для градуированного графа с конечным числом вершин на каждом уровне это свойство выполняется автоматически, поскольку предел вероятностных мер на конечном множестве --- снова вероятностная мера. Как мы увидим, для графа ${}_K\overline{\mathbb{YF}}$ это не всегда верно.

Ясно, что у последовательности $\tau'$ есть бесконечная подпоследовательность $\tau=\{(w_m,n_m)\}_{m=1}^\infty\in\left({}_K\overline{\mathbb{YF}}\right)^\infty,$ такая, что $n_m\to\infty$ и 
\begin{gather*}
  w_m = 1^{\beta_{k}(m)}21^{\beta_{k-1}(m)}2\ldots 21^{\beta_1(m)}21^{\beta_0(m)},  
\end{gather*}
где $k\in \overline{K},$ также для любого $i\in\overline{k}$ выполняется $\beta_i(m)\to c_i\in\mathbb{N}_0\cup\{\infty\},$ и кроме того,
\begin{gather*}
  \frac{n_m}{\#w_m+n_m}\to p=p(\tau)\in [0,1].  
\end{gather*}
Понятно, что для описания границы Мартина достаточно рассматривать только такие последовательности $\tau$. Рассмотрим два случая --- конечной и бесконечной суммы $\sum c_i$. Первый случай рассматривается в следующей лемме.
\begin{Lemma}\label{t1}
Пусть в вышеизложенных обозначениях $\tau\in\left( {}_K\overline{\mathbb{YF}}\right)^\infty$ --- это такая последовательность, что $\sum_{i=0}^k c_i<\infty.$ Тогда для любой вершины $(v,\ell)\in{}_K\overline{\mathbb{YF}}$ выполняется
\begin{equation*}
\mu_\tau(v,\ell)=
    \begin{cases}
        1,\quad\textup{если } v=\varepsilon, \\
        0,\quad\textup{иначе.}
    \end{cases}
\end{equation*}
\end{Lemma}
\begin{proof}
    В данном случае при достаточно больших $m$ выполняется $w_m=w = 1^{c_k}21^{c_{k-1}}2\ldots 21^{c_1}21^{c_0},$  
    а значит, в силу теоремы \ref{df} и \eqref{sim2}, при $\ell\in\mathbb{N}_0$ выполняется
\begin{align*}
    \mu_{\tau}(\varepsilon,\ell) = \lim_{m \to \infty} D(\varepsilon,\varepsilon,\ell)\frac{D(\varepsilon,w,n_m-\ell)}{D(\varepsilon,w,n_m)}=&\lim_{m \to \infty} \frac{ \sum_{i=0}^{k}F(\varepsilon,w,i)\cdot\binom{\#w+n_m-\ell-1}{|w|- i }}{ \sum_{i=0}^{k}F(\varepsilon,w,i)\cdot\binom{\#w+n_m-1}{|w|- i }}\\
    =&\lim_{m \to \infty} \frac{\sum_{i=0}^{k} F(\varepsilon,w,i)\cdot \frac{n_m^{|w|-i}}{(|w|- i)!}}{\sum_{i=0}^{k} F(\varepsilon,w,i)\cdot \frac{n_m^{|w|-i}}{(|w|- i)!}}=1.
\end{align*}
\end{proof}

Переходим к рассмотрению случая бесконечной суммы $\sum c_i.$ Далее положим $r(v,w):=\lim_{m\to\infty}r(v,w_m).$
\begin{Lemma}\label{lem1}
Пусть в вышеизложенных обозначениях $\tau\in\left( {}_K\overline{\mathbb{YF}}\right)^\infty$ --- это такая последовательность, что $\sum_{i=0}^k c_i=\infty.$ Тогда для любого $(v,\ell) \in {}_K\overline{\mathbb{YF}}$ выполняется
\begin{align}
  \mu_{\tau}(v,\ell)=&D(\varepsilon,v,\ell)\cdot p^\ell(1-p)^{|v|}\cdot\lim_{m\to\infty} \frac{ \sum_{i=0}^{r(v,w)}F(v,w_m,i)\cdot(1-p)^i \cdot p^{r(v,w)-i}}{\sum_{i=0}^{k}F(\varepsilon,w_m,i)\cdot(1-p)^i\cdot p^{k-i}}=\label{l21}\\
  =&D(\varepsilon,v,\ell)\cdot p^\ell(1-p)^{|v|}\cdot \lim_{m\to\infty}\frac{ Q_{v, w_m}(p)}{ Q_{\varepsilon,w_m}(p)}.\label{l22}
\end{align}
\end{Lemma}
\begin{proof} По теореме \ref{df}
\begin{gather}
    \lim_{m \to \infty} \frac{D(v,w_m,n_m-\ell)}{D(\varepsilon,w_m,n_m)}=\lim_{m \to \infty} \frac{\sum_{i=0}^{r(v,w)}F(v,w_m,i)\binom{s(v,w_m)+n_m-\ell-1}{|w_m|-|v|- i }}{ \sum_{i=0}^{k}F(\varepsilon,w_m,i)\binom{\#w_m+n_m-1}{|w_m|- i }}.\label{e4}
\end{gather}

   Далее несколько раз применим \eqref{sim1} и поймём, что при $i\in \overline{r(v,w)}$ и достаточно больших $m$
    \begin{align}
        \binom{s(v,w_m)+n_m-\ell-1}{|w_m|-|v|- i }=&\frac{(s(v,w_m)+n_m-\ell-1)!}{(|w_m|-|v|- i)!\cdot(-r(v,w)+n_m-\ell-1+i)! }\notag\\
        \sim&\frac{(|w_m|+n_m-1)!}{(|w_m|)!\cdot(n_m-1)! }\cdot \frac{\#w_m^{|v|+i}\cdot n_m^{r(v,w)+\ell-i}}{(\#w_m+n_m)^{r(v,w)+|v|+\ell}}\notag\\
        \sim&\frac{(|w_m|+n_m-1)!}{(|w_m|)!\cdot(n_m-1)! }\cdot (1-p)^{|v|+i}\cdot p^{r(v,w)+\ell-i}.\label{chis}
    \end{align}
    
    Аналогично при $i\in \overline{k}$ и достаточно больших $m$
    \begin{gather}
        \binom{\#w_m+n_m-1}{|w_m|- i }\sim\frac{(|w_m|+n_m-1)!}{(|w_m|)!\cdot(n_m-1)! }\cdot (1-p)^i\cdot p^{k-i}.\label{znam}
    \end{gather}
\eqref{l21} следует из \eqref{e4}, \eqref{chis} и \eqref{znam}, а \eqref{l22} следует из  леммы \ref{qf}. 
\end{proof}

Заметим, что если в условиях леммы \ref{lem1} $p(\tau)=0,$ то это будет вырожденная мера ($\mu_\tau(v,\ell)=\mathbbm{1}_{\ell=0}$), а если $p(\tau)=1,$ то это будет мера из леммы \ref{t1}. Обозначим данные меры через $\mu_0$ и $\mu_1$ соответственно. Будем рассматривать случай $p(\tau)\in(0,1).$

Далее пусть $f\in\overline{k}$ --- это такое число, что $\sum_{i=0}^{f-1} c_i<\infty,$ $c_f=\infty.$ Также пусть
\begin{gather*}
  \overline{w}=\overline{w}(\tau) = 21^{c_{f-1}}2\ldots 21^{c_1}21^{c_0}
\end{gather*}
(если $f=0,$ то $\overline{w}=\varepsilon$).

Из обозначения многочленов $Q$ несложно заметить, что при данной вершине $v\in{}_K\mathbb{YF},$ если существует $e\in\mathbb{N}_0,$ такое, что $v\le 1^e\overline{w},$ то
\begin{align*}
     Q_{v, w_m}(p) \sim \left(\prod_{j=0}^{k-f-1}\sum_{i=f}^{f+j}{\beta_i(m)}\right)\cdot  p^{k-f}\cdot \widetilde{Q}_{v,\overline{w}}(p),
\end{align*}
из чего следует, что
\begin{align}
\frac{Q_{v, w_m}(p)}{Q_{\varepsilon, w_m}(p)} \xrightarrow[m\to\infty]{} \frac{\widetilde{Q}_{v,\overline{w}}(p)}{Q_{\varepsilon, \overline{w}}(p)}.\label{qsim}
\end{align}

При данной вершине $W\in{}_K\mathbb{YF},$ не начинающейся на единицу, и данном $P\in(0,1)$ введём следующее обозначение. Для каждой вершины $(v,\ell) \in{}_K \overline{\mathbb{YF}}$ положим
\begin{gather*}
    \mu_{P,W}(v,\ell):=D(\varepsilon,v,\ell)\cdot P^\ell(1-P)^{|v|}\frac{ \widetilde{Q}_{v, W}(P) }{ Q_{\varepsilon,W}(P)}.
\end{gather*}

Из леммы \ref{lem1} и \eqref{qsim} следует
\begin{Zam}\label{zamech}
Пусть в вышеизложенных обозначениях $\tau\in\left( {}_K\overline{\mathbb{YF}}\right)^\infty$ --- это такая последовательность, что $\sum_{i=0}^k c_i=\infty,$ $p(\tau)=P\in(0,1),$ $\overline{w}(\tau)=W.$ Тогда для любой вершины $(v,\ell) \in {}_K\overline{\mathbb{YF}}$ такой, что существует $e\in\mathbb{N}_0,$ такое, что $v\le 1^eW,$ выполняется
\begin{gather*}
    \mu_{\tau}(v,\ell)=\mu_{P,W}(v,\ell).
\end{gather*}
\end{Zam}
Итак, мы готовы к формулировке и доказательству главной теоремы.
\begin{theorem}
Граница Мартина графа ${}_K\overline{\mathbb{YF}}$ состоит из мер $\mu_0,$ $\mu_1$ и мер вида $\mu_{P,W}$ при всех $W\in{}_K{\mathbb{YF}},$ не начинающихся на единицу, и $P\in(0,1).$ Причём все эти меры являются различными и все, кроме меры $\mu_0,$ являются вероятностными.
\end{theorem}
\begin{proof}
Для начала заметим, что при разных $W$ функции $\mu_{P,W}(v,\ell)$ принимают ненулевые значения на разных подмножествах ${}_K\overline{\mathbb{YF}},$ а при фиксированном $W$ заметим, что $\mu_{P,W}(\varepsilon,1)=P,$ из чего следует различность всех данных функций.

Далее зафиксируем $P,W$ и рассмотрим такую последовательность $\tau,$ что $p(\tau)= P,$ а также при каждом $m\in\mathbb{N}$ выполняется $w_m=1^{\beta(m)}W$ и
\begin{align*}
      \frac{|w_m|}{|w_m|+n_m}\le 1- P.
\end{align*}
Очевидно, что такая последовательность существует и что $\overline{w}(\tau)=W.$ 

В силу замечания \ref{zamech} при всех $(v,\ell)\in{}_K\overline{\mathbb{YF}}$ выполняется $\mu_\tau(v,\ell)=\mu_{P,W}(v,\ell),$ и для завершения доказательства теоремы достаточно доказать, что для любого натурального $\ell$ выполняется
\begin{gather*}
    \sum_{v\in{}_K\mathbb{YF}} \mu_{\tau}(v,\ell)=\sum_{v\in{}_K\mathbb{YF}} \mu_{P,W}(v,\ell)=1.  
\end{gather*}

Зафиксируем $\ell\in\mathbb{N}.$ Для применения теоремы Лебега о мажорируемой сходимости достаточно доказать, что существует такая константа $C,$ что для любой вершины $v\in{}_K{\mathbb{YF}}$ и для любого $m\in\mathbb{N}$ выполняется
 \begin{align}
    D(\varepsilon,v,\ell)\cdot \frac{D(v,w_m,n_m-\ell)}{D(\varepsilon,w_m,n_m)}\le C \cdot \mu_{P,W}(v,\ell). \label{major}
\end{align}

Понятно, что вершин $v,$ таких что $v\le W,$ лишь конечное число. Поэтому достаточно доказать существование константы $C$ для любой вершины вида $v=1^eW,$ $e\in\mathbb{N}$ и любого $m\in\mathbb{N}$ такого, что $v\le w_m,$ $\ell\le n_m.$ При таких условиях по теореме \ref{df}
\begin{gather*}
     \frac{D(v,w_m,n_m-\ell)}{D(\varepsilon,w_m,n_m)}=\frac{\binom{|w_m|-|v|+n_m-\ell-1}{|w_m|-|v|}}{\sum_{i=0}^{d(W)}F(\varepsilon,W,i)\binom{\#w_m+n_m-1}{|w_m|- i}}.\notag
\end{gather*}

Согласно \eqref{znam} и предложению \ref{z1}, существует константа $C_1>0$ такая, что знаменатель данной дроби оценивается снизу как 
\begin{gather*}
    C_1\frac{(|w_m|+n_m)!}{(|w_m|)!\cdot(n_m)!}.
\end{gather*}

Теперь оценим числитель сверху:
    \begin{align*}
        \binom{|w_m|-|v|+n_m-\ell-1}{|w_m|-|v|}\le & \binom{|w_m|-|v|+n_m}{|w_m|-|v|}\\
        =&\frac{(|w_m|+n_m)!}{(|w_m|)!\cdot(n_m)! }\cdot \frac{(|w_m|-|v|+1)\cdots (|w_m|)}{(|w_m|+n_m-|v|+1)\cdots (|w_m|+n_m)}\\
        \le&\frac{(|w_m|+n_m)!}{(|w_m|)!\cdot(n_m)! }\cdot \frac{(|w_m|)^{|v|}}{(|w_m|+n_m)^{|v|}}\\
        \le&\frac{(|w_m|+n_m)!}{(|w_m|)!\cdot(n_m)! }\cdot (1-P)^{|v|}.
    \end{align*}

Таким образом, левая часть \eqref{major} не превосходит 
\begin{align*}
    D(\varepsilon,v,\ell)\cdot\frac{(1-P)^{|v|}}{C_1}=C\cdot D(\varepsilon,v,\ell)\cdot P^\ell(1-P)^{|v|}\frac{1}{ Q_{\varepsilon,W}(P)}=C\cdot\mu_{P,W}(v,\ell)
\end{align*}
при $C=Q_{\varepsilon,W}(P)/(C_1\cdot P^\ell).$

Теорема доказана.
\end{proof}

\section{Благодарности}
Я очень признателен моему научному руководителю Фёдору Владимировичу Петрову за постановку задачи, проявленное терпение при проверке доказательства и редактировании статьи, а также за моральную поддержку на протяжении всего периода работы. 

Работа выполнена при поддержке СПбГУ, шифр проекта 159547043.

\selectlanguage{english}

\title{Martin boundary of the jump graph on subgraphs of the Young--Fibonacci graph}

\author{V.~Yu.~Evtushevsky}

\maketitle

\setcounter{page}{1}

\etocsettagdepth{en}{all}
\etocsettagdepth{ru}{none}

\tableofcontents

\theoremstyle{style2}
\newtheorem{ENLemma}{Lemma}

\theoremstyle{style2}
\newtheorem{ENAlg}{Algorithm}

\theoremstyle{style2}
\newtheorem{ENCol}{Corollary}

\theoremstyle{style2}
\newtheorem{ENtheorem}{Theorem}

\theoremstyle{style2}
\newtheorem{ENDef}{Definition}

\theoremstyle{style2}
\newtheorem{ENProp}{Proposition}

\theoremstyle{style2}
\newtheorem{ENProblem}{Problem}

\theoremstyle{style2}
\newtheorem{ENZam}{Remark}

\theoremstyle{style2}
\newtheorem{ENOboz}{Notation}

\theoremstyle{style2}
\newtheorem{ENEx}{Example}

\theoremstyle{style2}
\newtheorem{ENNab}{Observation}

\setcounter{section}{0}
\setcounter{figure}{0} 
\setcounter{equation}{0}

\begin{abstract}
For the jump graph on the subgraph of the Young--Fibonacci graph formed by words with at most $K$ twos, formulas for the number of paths between vertices are obtained and the Martin boundary of the path space is described.
\end{abstract}

\etocdepthtag.toc{en}
\section{Introduction}

Let $G=(V,E)$ be an acyclic directed graph. It corresponds to a partially ordered set $(V,\le)$ with the reachability relation: $v\le w$ if there is a path from $v$ to $w$. Consider a graded graph with a root at level zero, whose $i$th level for $i>0$ is $V\times \{i\},$ and with an edge from $(v,i)$ to $(w,i+1)$ if $v\leqslant w$. It is natural to call it the jump graph on $G$. Natural questions about a graded graph are the enumeration of paths between vertices (in this case, the enumeration of weakly increasing chains of a given length between two elements of the partially ordered set $(V,\le)$) and the description of central probability measures on the space of paths from the root to infinity (centrality means that all initial segments of a random path with a fixed vertex at the $n$th level are equally likely). For the graded Young and Young--Fibonacci graphs, these problems have long been solved. For jump graphs, the situation is more complicated, although in the case of the Young graph one can write a determinant formula for the number of paths between vertices. The present work is devoted to finding the Martin boundary of the jump graph on the natural subgraph of the Young--Fibonacci graph formed by words with a bounded number of twos (see the detailed definitions below).

\begin{figure}[htbp]
    \centering
    \includegraphics[width=12cm, height=10cm]{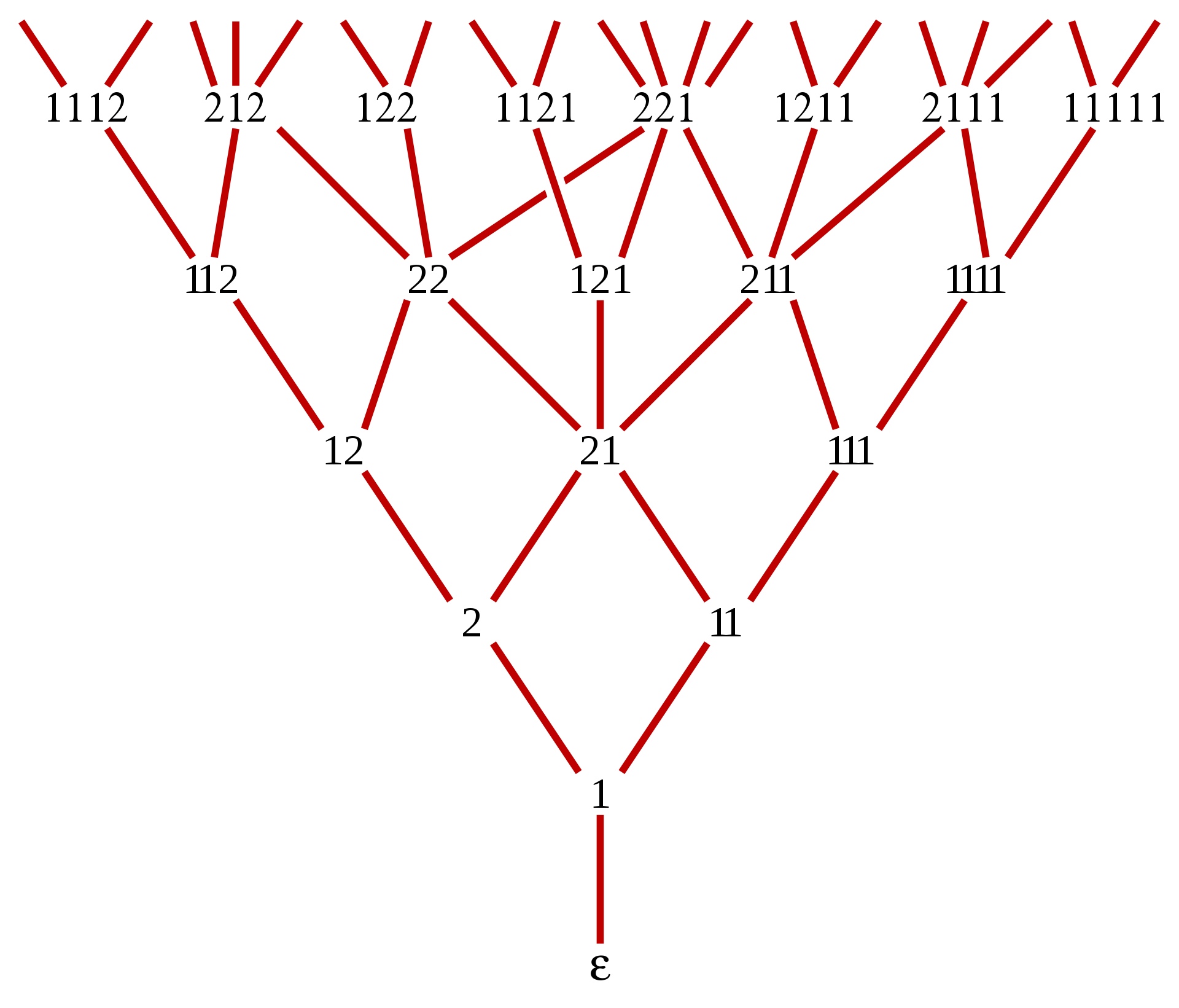}
    \caption{The Young--Fibonacci graph}
    \label{ENyf}
\end{figure}

Consider words over the alphabet $\{1,2\}$ with a given digit sum $n.$ As is well known, their number is the Fibonacci number $F_{n+1}$ ($F_0=0,$ $F_1=1,$ $F_{k+2}=F_{k+1}+F_k$), and this is the most common combinatorial interpretation of Fibonacci numbers. One can also think of tilings of a $2\times n$ strip with dominoes of sizes $1\times 2$ and $2\times 1,$ associating twos with pairs of horizontal dominoes and ones with vertical dominoes.

Introduce a partial order on this set of words: we say that a word $v$ \emph{precedes} a word $w$ if, after removing the longest common suffix, the word $w$ has at least as many twos remaining as the word $v$ has digits remaining. 

This is indeed a partial order; moreover, the corresponding partially ordered set is a modular lattice known as the Young--Fibonacci lattice.

\renewcommand{\labelenumi}{\arabic{enumi}.}

The Hasse diagram of this lattice is called the Young--Fibonacci graph (it is shown in Fig. \ref{ENyf}). This is a graded graph, which we visualize as growing from bottom to top, starting with the empty word. The grading is given by the digit sum function. Let us explicitly describe the directed edges. The ``upward'' edges from a given word $v$ lead to words obtained from $v$ by one of two operations:
\begin{enumerate}
    \item replace the leftmost one with a two;
    \item insert a one in a position
    such that there are no ones to its left.
\end{enumerate}

In this work, we regard the Young--Fibonacci graph as directed and containing only ``upward'' edges.

In addition to being modular, this graph is $1$-differential, that is, for each vertex the out-degree exceeds the in-degree by $1$.

The study of the graded Young--Fibonacci graph was initiated in 1988 simultaneously and independently by Richard Stanley \cite{ENSt} and Sergei Vladimirovich Fomin \cite{ENFo}.

One reason for the interest in it is that there are only two rooted locally finite $1$-differential modular lattices; the other is the lattice of Young diagrams, which plays a key role in the representation theory of the symmetric group \cite[Theorem~1.3]{ENBur}. 

The key questions about graded graphs concern central measures on the space of (infinite) paths in the graph. This viewpoint was systematically developed in the works of Anatoly Moiseevich Vershik; we refer the reader to his recent survey \cite{ENVer} and the literature cited there.

Among central measures, one singles out those that are limits of measures induced by paths to distant vertices---the so-called \emph{Martin boundary} of the graph.

The boundary of the path space of the Young--Fibonacci graph was studied in the work of Frederick Goodman and Sergei Vasilievich Kerov (2000) \cite{ENGK}. They used Okada's algebraic formalism \cite{ENOkada}.

As follows from the definition itself, the asymptotic question about the boundary is directly related to the enumerative question about the number of paths between two vertices of the graph. We note Fomin's important general work \cite{ENFo1} on the enumeration of paths in graded graphs, which gives a number of general identities and points out a connection, among other things, with a generalization of the Robinson--Schensted--Knuth algorithm.

In the author's work \cite{ENEvtuh1} (a short version of which was published as \cite{ENEvtuh2}), an explicit formula was proposed for the number of paths between two vertices in the Young--Fibonacci graph, and in the work of Ivan Alekseevich Bochkov and the author \cite{ENBE}, for all measures of the Martin boundary of the Young--Fibonacci graph, an explicit formula was found for the measure of the set of all infinite paths from the root passing through a given vertex.

In Sections 2 and 3, we propose formulas for the number of paths between two vertices in the jump graph on the Young--Fibonacci graph. In Section 5, we consider the subgraph of the Young--Fibonacci graph formed by words with at most $K$ twos for an arbitrary nonnegative integer $K.$ It can be regarded as an analogue of the subgraph of the Young graph formed by diagrams with at most $K$ rows. We completely describe the Martin boundary of the jump graph on this graph, and for each measure of this boundary we give a formula for the measure of the set of all infinite paths from the root passing through a given vertex.

\section{Basic concepts and notation}

The sets of nonnegative integers and positive integers (natural numbers) are denoted by $\mathbb{N}_0$ and $\mathbb{N}$, respectively. For $n\in \mathbb{N}_0$, set $\overline{n}:=\{0,1,\ldots,n\}.$

Let $\mathbb{YF}$ be the Young--Fibonacci graph; its vertices are finite words in $1$ and $2.$ Hereafter, $\mathbb{YF}$ as a set is precisely the set of vertices of the Young--Fibonacci graph. For $v,w\in\mathbb{YF}$, the word $vw$ is defined as the concatenation of $v$ and $w.$ Also, for $v\in\mathbb{YF}$ and $n\in\mathbb{N}_0$, the word $v^n$ is defined in the natural way.

For $v \in \mathbb{YF}$, denote the digit sum of the word $v$ by $|v|$. The number $|v|$ is the \emph{rank} of the vertex $v$ in the graph $\mathbb{YF}$ (we assume that the empty word has rank $0$ and that moving upward along an edge increases the rank by $1$). In this work, we regard the graph $\mathbb{YF}$ as directed, with edges from vertices of lower rank to vertices of higher rank.

For $v,w\in \mathbb{YF}$, denote the number of paths from $v$ to $w$ by $d(v,w).$ In particular, $d(v,w)=0$ when $|w|<|v|.$

Denote the number of digits in $v\in\mathbb{YF}$ by $\#v,$ and the number of twos in $v$ by $d(v).$ Clearly, $\#v+d(v)=|v|.$ 

Also, let $v,w \in \mathbb{YF}.$ The vertices $v,w$ have a longest common suffix $u\in \mathbb{YF},$ that is, the longest word for which $v=v'u,$ $w=w'u.$ Set $w_v:=w',$ $v_w:=v',$ $r(v,w):=d(w_v)-\# v_w$. It follows from the definition of the Young--Fibonacci graph that  
\begin{align*}
    r(v,w)\ge 0&\Longleftrightarrow d(v,w)\ge 1;\\
    r(v,w)= 0&\Longrightarrow d(v,w)=1.
\end{align*}
When $r(v,w)\ge 0$, we write $v\le w,$ when $r(v,w)= 0$, we write $v\lesssim w,$ and when $r(v,w)> 0$, we write $v\lnsim w.$   

Thus, let us define the directed graded graph $\overline{\mathbb{YF}}$ (the \emph{jump graph on the Young--Fibonacci graph}). Its vertex set is $\overline{\mathbb{YF}}:=\{(\varepsilon,0)\}\cup \mathbb{YF}\times \mathbb{N},$ with exactly one vertex (the root) at level $0$, namely $(\varepsilon,0),$ while, for each fixed $\ell\in\mathbb{N}$, level $\ell$ contains all vertices of the form $(v,\ell),$ where $v\in\mathbb{YF}.$ A directed edge between vertices $(v,\ell),(w,\ell+1)\in \overline{\mathbb{YF}}$ is drawn if and only if $v\le w.$ 

If $(v,\ell), (w,\ell+n)\in \overline{\mathbb{YF}},$ denote the number of paths from $(v,\ell)$ to $(w,\ell+n)$ by $D(v,w,n).$ In particular, $D(v,w,n)=0$ when $n<0$ or when $v\nleqslant w.$ The following relations are also evident:
\begin{align}
   \text{For } n\in\mathbb{N}_0 \qquad D(v,\varepsilon,n)=&\mathbbm{1}_{v=\varepsilon};\notag \\
   D(v,w,0)=&\mathbbm{1}_{v=w}; \notag \\
   \text{For } n\in\mathbb{N} \qquad D(v,w,n)=&\sum_{u\le w}D(v,u,n-1);\notag\\
   \text{For } v\lesssim w,\; n\in\mathbb{N}_0\qquad   D(v,w,n) =& \binom{|w|-|v|+n-1}{|w|-|v|}. \label{ENr=0}
\end{align}

\section{First formula for $D(v,w,n)$}

It follows from the definition of the Young--Fibonacci graph that, for any vertex $w\in\mathbb{YF}$,
\begin{gather*}
\{u:\; u\le 1w\}= \{u:\; u\le w\}\cup \{1w\}; \qquad \{u:\; u\le 2w\} = \{\varepsilon\} \cup\{1u:\; u\le w\}\cup \{2u:\; u\le w\} .
\end{gather*}
Thus, it is evident that for any $v,w\in\mathbb{YF},$ $n\in\mathbb{N}$
\begin{align}
D(v,1w,n)=&D(v,w,n)+D(v,1w,n-1);\label{ENe1}\\
D(v,2w,n)=&\mathbbm{1}_{v=\varepsilon}+\sum_{u\le w}(D(v,1u,n-1)+D(v,2u,n-1)).\label{ENe2}
\end{align}

\begin{ENtheorem}\label{ENrf}
Let $v,w\in {\mathbb{YF}},$ $n\in\mathbb{N}_0.$ Then
\begin{align}
        D(v,1w,n)=& \mathbbm{1}_{v=1w}+ \sum_{m=1}^n D(v,w,m);\label{ENrf1}\\
        D(\varepsilon,2w,n)=&\sum_{m=1}^n m\cdot D(\varepsilon,w,m);\label{ENrf2}\\
        D(1v,2w,n)=&\sum_{m=1}^n D(v,w,m)+\sum_{m=1}^n (m-1)\cdot D(1v,w,m)\label{ENrf3};\\
        D(2v,2w,n)=&D(v,w,n)+\sum_{m=1}^n (m-1)\cdot D(2v,w,m).\label{ENrf4}
\end{align}
\end{ENtheorem}
\begin{proof}
    We prove the theorem by induction on $n.$ The base case $n=0$ is evident for all the equalities. The induction step to $n$ for equality \eqref{ENrf1} follows immediately from formula \eqref{ENe1}. Next, we prove the induction step to $n$ for equalities \eqref{ENrf2},\eqref{ENrf3},\eqref{ENrf4} in succession. In each case, we use formula \eqref{ENe2}, the already proved equality \eqref{ENrf1}, and the induction hypothesis:
   \begin{align*}
    D(\varepsilon,2w,n)=&1+\sum_{u\le w}(D(\varepsilon,1u,n-1)+D(\varepsilon,2u,n-1))=\\
    =&1+\sum_{u\le w}\left(\sum_{m=1}^{n-1} D(\varepsilon,u,m)+\sum_{m=1}^{n-1} m\cdot D(\varepsilon,u,m)  \right)\\=&1+\sum_{m=2}^{n} m\cdot D(\varepsilon,w,m)=\sum_{m=1}^{n} m\cdot D(\varepsilon,w,m);\\
    D(1v,2w,n)=&\sum_{u\le w}(D(1v,1u,n-1)+D(1v,2u,n-1))\\
    =&\sum_{u\le w}\left(\mathbbm{1}_{v=u} +\sum_{m=1}^{n-1} D(1v,u,m)+\sum_{m=1}^{n-1}  D(v,u,m)+\sum_{m=1}^{n-1} (m-1)\cdot D(1v,u,m)  \right)\\
    =&D(v,w,1)+\sum_{m=2}^{n}  D(v,w,m)+\sum_{m=2}^{n} (m-1)\cdot D(1v,w,m)\\
    =&\sum_{m=1}^n D(v,w,m)+\sum_{m=1}^n (m-1)\cdot D(1v,w,m);\\
    D(2v,2w,n)=&\sum_{u\le w}(D(2v,1u,n-1)+D(2v,2u,n-1))\\
    =&\sum_{u\le w}\left(\sum_{m=1}^{n-1} D(2v,u,m)+  D(v,u,n-1)+\sum_{m=1}^{n-1} (m-1)\cdot D(2v,u,m)  \right)\\
    =&D(v,w,n)+\sum_{m=1}^n (m-1)\cdot D(2v,w,m).\\
    \end{align*}
\end{proof}
\section{Second formula for $D(v,w,n)$}
For $a,b,n\in\mathbb{N}_0,$ such that $a\le b,$ the following identities hold:
\begin{align}
    \sum_{m=1}^n \binom{a+m-1}{b}=&\binom{a+n}{b+1};\label{ENf4}\\
    \sum_{m=1}^n m\binom{a+m-1}{b}=&(b+1)\binom{a+n}{b+2}+(b+1-a)\binom{a+n}{b+1};\label{ENf5}\\
    \sum_{m=1}^n (m-1)\binom{a+m-1}{b}=&(b+1)\binom{a+n}{b+2}+(b-a)\binom{a+n}{b+1}.\label{ENf6}
\end{align}    
They are verified by induction on $n$ and will be used below.

\renewcommand{\labelenumi}{\arabic{enumi}$^\circ$}

We also need new notation. For $w,v\in\mathbb{YF}$, set $s(v,w):=\#(w_v)-d(v_w).$ We immediately note that $r(v,w)+s(v,w)=|w|-|v|,$ and also that, when $v\le w$, we have $s(v,w)\ge 0.$

The following function $F$ allows us to find a new formula for $D(v,w,n).$

The function 
\begin{gather*}
  F: \quad \{(v,w,i)\in \mathbb{YF} \times \mathbb{YF}\times \mathbb{Z}:\; v\le w\} \to \mathbb{Z}  
\end{gather*}
is defined recursively as follows:

\begin{itemize}
    \item If $v\lesssim w,$ then
    \begin{gather*}
      F(v,w,i)=\mathbbm{1}_{i=0};  
    \end{gather*}
    \item If $v\lnsim 1w,$ then 
    \begin{gather*}
        F(v,1w,i)=F(v,w,i);
    \end{gather*}
    \item    
    \begin{gather*}
    F(\varepsilon,2w,i)=(|w|+1-i) \cdot F(\varepsilon,w,i) + (d(w)+2-i) \cdot F(\varepsilon,w,i-1);    
    \end{gather*} 
    \item If $v\lnsim 2w,$ $v=\alpha v',$ $\alpha\in\{1,2\},$ then
        \begin{align*}
            F(v,2w,i)=&F(v',w,i)+\mathbbm{1}_{v,w}\cdot F(v',w,i-1)+(|w|-|v|+1-i) \cdot F(v,w,i) \\&+ (r(v,w)+1-i) \cdot F(v,w,i-1),
        \end{align*}
        where
    \begin{gather*}
    \mathbbm{1}_{v,w}=
    \begin{cases}
    1,\quad \text{if } \alpha=1,\;v_w=\varepsilon,\\
    0,\quad \text{otherwise.}
    \end{cases}
\end{gather*}
\end{itemize}

\begin{ENProp}\label{ENz1}
For given $v\le w$, the values of the function $F(v,w,i)$ are positive integers when $i\in \overline{r(v,w)}$ and zero otherwise. 
\end{ENProp}
\begin{proof}
    This proposition is easily proved by induction, using the fact that if $v\le w,$ then
\begin{align*}
    r(v,2w)=&r(v,w)+1;\\
    r(v,2w)=&r(v',w)+\mathbbm{1}_{v,w}.
\end{align*}    
\end{proof}

\begin{ENtheorem}\label{ENdf}
Let $n\in\mathbb{N}_0,$ $v,w\in\mathbb{YF}:$ $v\le w.$ Then
\begin{gather*}
  D(v,w,n)=\sum_{i=0}^{r(v,w)}F(v,w,i)\binom{s(v,w)+n-1}{|w|-|v|- i}. 
\end{gather*}
\end{ENtheorem}
\begin{proof}
We prove the theorem by induction on $r(v,w),$ and, for a given value of $r(v,w)$, by induction on $|w|.$ 

In the proof, we use the fact that, by Proposition \ref{ENz1}, for any $n\in\mathbb{N}_0,$ $v,w\in\mathbb{YF}:$ $v\le w$, we have
\begin{gather*}
    \sum_{i=0}^{r(v,w)}F(v,w,i)\binom{s(v,w)+n-1}{|w|-|v|- i }=\sum_{i=-\infty}^{\infty}F(v,w,i)\binom{s(v,w)+n-1}{|w|-|v|- i}.\label{ENinfinity}
\end{gather*}

The base case $v\lesssim w$ reduces to \eqref{ENr=0}. We prove the induction step to $(v,\gamma w)$ for $\gamma\in\{1,2\},$ $v\lnsim\gamma w.$

In the case $\gamma=1$, we use \eqref{ENrf1}, the induction hypothesis, and \eqref{ENf4}:
\begin{align*}
    D(v,1w,n)=\sum_{m=1}^{n} D(v,w,m)=&\sum_{i=0}^{r(v,w)}F(v,w,i)\sum_{m=1}^n\binom{s(v,w)+m-1}{|w|-|v|-i}\\
    =&\sum_{i=-\infty}^{\infty}F(v,1w,i)\binom{s(v,1w) +n-1}{|1w|-|v|-i}.
\end{align*}

In the case $v=\varepsilon$, we use \eqref{ENrf2}, the induction hypothesis, and \eqref{ENf5}:
\begin{align*}
    D(\varepsilon,2w,n)=&\sum_{m=1}^n m\cdot D(\varepsilon,w,m)=\sum_{i=0}^{d(w)}F(\varepsilon,w,i) \sum_{m=1}^n m \binom{\#w+m-1}{|w|-i}\\
    =&\sum_{i=-\infty}^{\infty}F(\varepsilon,w,i) \left( (|w|+1-i) \binom{\#(2w)+n-1}{|2w|-i} + (d(w)+1-i) \binom{\#(2w)+n-1}{|2w|-1-i} \right)\\
    =&\sum_{i=-\infty}^{\infty}F(\varepsilon,2w,i)\binom{\#(2w)+n-1}{|2w|-i}.   
\end{align*}

The more substantial case is $\gamma=2,$ $v\ne \varepsilon.$ Note that, by the induction hypothesis and \eqref{ENf6},
\begin{multline}
    \sum_{m=1}^n  (m-1)\cdot D(v,w,m)=\sum_{i=0}^{r(v,w)} F(v,w,i)\sum_{m=1}^n (m-1)\binom{s(v,w) +m-1}{|w|-|v|-i}\\
    =\sum_{i=-\infty}^{\infty}F(v,w,i) \biggl( (|w|-|v|+1-i) \binom{s(v,2w)+n-1}{|2w|-|v|-i}+(r(v,w)-i) \binom{s(v,2w)+n-1}{|2w|-|v|-1-i} \biggr).\label{ENe3}
\end{multline}

Next, consider two cases:

\renewcommand{\labelenumi}{\arabic{enumi}$^\circ$}

\begin{enumerate}
    \item Let $v=1v'.$ Note that, by the induction hypothesis, \eqref{ENf4}, and the fact that in this case $s(v',w)=s(v,2w)-1+\mathbbm{1}_{v,w},$ we have
    
    \begin{align*}
    \sum_{m=1}^n  D(v',w,m)=&\sum_{i=0}^{r(v',w)}F({v'},w,i)\sum_{m=1}^n\binom{s(v',w)+m-1}{|w|-|v'|- i }\\
    =&\sum_{i=-\infty}^{\infty}F({v'},w,i)\binom{s(v,2w)+n-1+\mathbbm{1}_{v,w}}{|2w|-|v|- i }\\
    =&\sum_{i=-\infty}^{\infty}F({v'},w,i)\binom{s(v,2w)+n -1}{|2w|-|v|- i}
    +\mathbbm{1}_{v,w}\cdot\sum_{i=-\infty}^{\infty}F({v'},w,i)\binom{s(v,2w)+n -1}{|2w|-|v|-1- i}.
    \end{align*}
    Add this equality to equality \eqref{ENe3}, apply \eqref{ENrf3}, and obtain the desired result.
    \item Let $v=2v'.$ Note that, by the induction hypothesis and the fact that in this case $s(v',w)=s(v,2w),$ we have
\begin{align*}
    D(v',w,n)=\sum_{i=-\infty}^{\infty}F(v',w,i)\binom{s(v,2w)+n-1}{|2w|-|v|-i}.
\end{align*}
    Add this equality to equality \eqref{ENe3}, apply \eqref{ENrf4}, and obtain the desired result.
\end{enumerate}

\end{proof}

\section{The polynomials $Q_{v,w}(p)$}

The following polynomials $Q_{v,w}(p)$ are closely related to the function $F(v,w,n)$ and will be used in what follows.

For given $v,w\in\mathbb{YF}:$ $v\le w$, the polynomial $Q_{v,w}(p)$ is defined recursively as follows:
\begin{itemize}
    \item If $v\lesssim w,$ then
    \begin{gather*}
       Q_{v,w}(p)=1; 
    \end{gather*}
    \item If $v\lnsim 1w,$ then
    \begin{gather*}
      Q_{v,1w}(p)=Q_{v,w}(p);  
    \end{gather*}
    \item 
    \begin{gather*}
        Q_{\varepsilon,2w}(p)=(1-p)\cdot(p\cdot Q_{\varepsilon,w}(p))'+(|w|+1)\cdot p \cdot Q_{\varepsilon,w}(p);
    \end{gather*}
    \item If $v\lnsim 2w,$ $v=\alpha v',$ $\alpha\in\{1,2\},$ then
    \begin{gather*}
      Q_{v,2w}(p)=p\cdot(1-p)\cdot Q_{v,w}'(p)+(|w|-|v|+1)\cdot p \cdot Q_{v,w}(p)+Q_{v',w}(p).  
    \end{gather*}
\end{itemize}

\begin{ENLemma}\label{ENqf}
Let $v,w\in\mathbb{YF}:$ $v\le w$. Then
\begin{gather*}
    Q_{v,w}(p)=\sum_{i=0}^{r(v,w)}F(v,w,i)\cdot(1-p)^i\cdot p^{r(v,w)-i}.
\end{gather*}
In particular, for every $p\in(0,1)$, we have $Q_{v,w}(p)>0.$
\end{ENLemma}
\begin{proof}
By Proposition \ref{ENz1}, this lemma is equivalent to the assertion that, for $v,w\in\mathbb{YF}:$ $v\le w$ 
\begin{gather}
    Q_{v,w}(p)=\sum_{i=-\infty}^{\infty}F(v,w,i)\cdot(1-p)^i\cdot p^{r(v,w)-i}.\label{ENqinf}
\end{gather}

We shall prove \eqref{ENqinf} by induction on $r(v,w),$ and, for a given value of $r(v,w)$, by induction on $|w|.$ The base case $v\lesssim w$ is obvious. We shall prove the induction step to $(v,\gamma w)$ for $\gamma\in\{1,2\},$ $v\lnsim\gamma w.$ We immediately note that the case $\gamma=1$ is obvious. Thus, let $\gamma=2.$ Observe that, by the induction hypothesis,
\begin{gather}
     (|w|-|v|+1)\cdot p\cdot Q_{v,w}(p)=(|w|-|v|+1)\cdot\sum_{i=-\infty}^{\infty}F(v,w,i)\cdot(1-p)^i\cdot p^{r(v,2w)-i}.\label{ENe5}
\end{gather}

Next, consider two cases:
\begin{enumerate}
    \item Let $v=\varepsilon.$ In this case, by the induction hypothesis,
\begin{multline}
        (1-p)\cdot (p\cdot Q_{\varepsilon,w}(p))'= (1-p)\cdot\left(\sum_{i=-\infty}^{\infty}F(\varepsilon,w,i)\cdot(1-p)^i\cdot p^{d(w)+1-i}\right)'\\
        =\sum_{i=-\infty}^{\infty}F(\varepsilon,w,i)(-i\cdot(1-p)^{i}\cdot p^{d(2w)-i}+(d(w)+1-i)\cdot(1-p)^{i+1}\cdot p^{d(2w)-1-i}).\label{ENeps}
\end{multline}
Adding equalities \eqref{ENe5}, \eqref{ENeps}, we obtain the desired result.
    \item Let $v\ne \varepsilon.$ In this case, by the induction hypothesis,
    \begin{multline}
     p\cdot (1-p)\cdot Q'_{v,w}(p)=p\cdot (1-p)\cdot\left( \sum_{i=-\infty}^{\infty}F(v,w,i)\cdot(1-p)^i\cdot p^{r(v,w)-i}\right)'\\
     =\sum_{i=-\infty}^{\infty}F(v,w,i)(-i\cdot(1-p)^{i}\cdot p^{r(v,2w)-i}+(r(v,w)-i)\cdot(1-p)^{i+1}\cdot p^{r(v,2w)-1-i});\label{ENq1}
     \end{multline}
     \begin{multline}
     Q_{v',w}(p)=\sum_{i=-\infty}^{\infty}F(v',w,i)\cdot(1-p)^i\cdot p^{r(v,2w)-\mathbbm{1}_{v,w}-i}\\
     =\sum_{i=-\infty}^{\infty}F(v',w,i)\cdot(1-p)^i\cdot p^{r(v,2w)-i}+\mathbbm{1}_{v,w}\sum_{i=-\infty}^{\infty}F(v',w,i)\cdot(1-p)^{i+1}\cdot p^{r(v,2w)-1-i}.\label{ENq2}
    \end{multline}
    Adding equalities \eqref{ENe5}, \eqref{ENq1}, \eqref{ENq2}, we obtain the desired result.
\end{enumerate}
\end{proof}

In what follows, we shall need the following notation for any $v,w\in\mathbb{YF}:$
\begin{equation*}
    \widetilde{Q}_{v,w}(p)=
    \begin{cases}
        Q_{v,w}(p), &\text{if } v\le w,\\
        1,&\text{if } \exists e\in\mathbb{N}:\; v = 1^ew,\\ 
        0,&\text{otherwise.} 
    \end{cases}
\end{equation*}

\section{Description of the Martin boundary}
Observe that if $a\in\mathbb{Z},$ $b\in\mathbb{N}_0,$ and the positive integer parameter $n$
tends to infinity, then
\begin{align}
    \frac{(n+a)!}{n!}\sim& n^a;\label{ENsim1}\\
    \binom{n+a}{b}\sim&\frac{n^b}{b!}\label{ENsim2}.
\end{align}
These facts will be used in the proofs below. 

Fix a nonnegative integer $K.$ 

Next, let ${}_K\mathbb{YF}$ be the induced subgraph of the graph $\mathbb{YF}$ on vertices of the form $v\in\mathbb{YF},$ such that $d(v)\le K,$ and let ${}_K\overline{\mathbb{YF}}$ be the induced subgraph of the graph $\overline{\mathbb{YF}}$ on vertices of the form $(v,n)\in\overline{\mathbb{YF}},$ such that $v\in{}_K{\mathbb{YF}}.$

Consider a sequence of vertices $\tau'=\{(w'_m,n'_m)\}_{m=1}^\infty\in\left({}_K\overline{\mathbb{YF}}\right)^\infty,$ such that $n'_m\to\infty.$ If, for every vertex $(v,\ell)\in{}_K\overline{\mathbb{YF}}$, the limit
\begin{gather*}
  \mu_{\tau'}(v,\ell):=\lim_{m\to \infty} D(\varepsilon,v,\ell)\frac{D(v,w'_m,n'_m-\ell)}{D(\varepsilon,w'_m,n'_m)},  
\end{gather*}
exists, then the sequence $\tau'$ generates a measure $\mu_{\tau'}$ on the space of infinite paths in the graph ${}_K\overline{\mathbb{YF}},$ with the following property: for the set of all paths passing through the vertex $(v,\ell),$ the measure equals $\mu_{\tau'}(v,\ell),$ while all $D(\varepsilon,v,\ell)$ initial path segments from the root to $(v,\ell)$ have the same measure $\mu_{\tau'}(v,\ell)/D(\varepsilon,v,\ell).$ These measures form the so-called \emph{Martin boundary} of the graph ${}_K\overline{\mathbb{YF}}.$ Thus, our task is to describe all such measures. In addition, the measure $\mu_{\tau'}$ is called a \emph{probability measure} if, for every $\ell\in\mathbb{N}_0$, we have
\begin{gather*}
    \sum_{v\in{}_K\mathbb{YF}} \mu_{\tau'}(v,\ell)=1.
\end{gather*}
We also want to determine which measures of the Martin boundary of the graph ${}_K\overline{\mathbb{YF}}$ are probability measures. For a graded graph with finitely many vertices at each level, this property holds automatically, since the limit of probability measures on a finite set is again a probability measure. As we shall see, this is not always true for the graph ${}_K\overline{\mathbb{YF}}$.

Clearly, the sequence $\tau'$ has an infinite subsequence $\tau=\{(w_m,n_m)\}_{m=1}^\infty\in\left({}_K\overline{\mathbb{YF}}\right)^\infty,$ such that $n_m\to\infty$ and 
\begin{gather*}
  w_m = 1^{\beta_{k}(m)}21^{\beta_{k-1}(m)}2\ldots 21^{\beta_1(m)}21^{\beta_0(m)},  
\end{gather*}
where $k\in \overline{K},$ and, for every $i\in\overline{k}$, we also have $\beta_i(m)\to c_i\in\mathbb{N}_0\cup\{\infty\},$ and, moreover,
\begin{gather*}
  \frac{n_m}{\#w_m+n_m}\to p=p(\tau)\in [0,1].  
\end{gather*}
It is clear that, to describe the Martin boundary, it suffices to consider only such sequences $\tau$. Consider two cases, according to whether the sum $\sum c_i$ is finite or infinite. The first case is considered in the following lemma.
\begin{ENLemma}\label{ENt1}
In the notation introduced above, let $\tau\in\left( {}_K\overline{\mathbb{YF}}\right)^\infty$ be a sequence such that $\sum_{i=0}^k c_i<\infty.$ Then, for every vertex $(v,\ell)\in{}_K\overline{\mathbb{YF}}$, we have
\begin{equation*}
\mu_\tau(v,\ell)=
    \begin{cases}
        1,\quad\textup{if } v=\varepsilon, \\
        0,\quad\textup{otherwise.}
    \end{cases}
\end{equation*}
\end{ENLemma}
\begin{proof}
    In this case, for sufficiently large $m$, we have $w_m=w = 1^{c_k}21^{c_{k-1}}2\ldots 21^{c_1}21^{c_0},$  
    and hence, by Theorem \ref{ENdf} and \eqref{ENsim2}, for $\ell\in\mathbb{N}_0$, we have
\begin{align*}
    \mu_{\tau}(\varepsilon,\ell) = \lim_{m \to \infty} D(\varepsilon,\varepsilon,\ell)\frac{D(\varepsilon,w,n_m-\ell)}{D(\varepsilon,w,n_m)}=&\lim_{m \to \infty} \frac{ \sum_{i=0}^{k}F(\varepsilon,w,i)\cdot\binom{\#w+n_m-\ell-1}{|w|- i }}{ \sum_{i=0}^{k}F(\varepsilon,w,i)\cdot\binom{\#w+n_m-1}{|w|- i }}\\
    =&\lim_{m \to \infty} \frac{\sum_{i=0}^{k} F(\varepsilon,w,i)\cdot \frac{n_m^{|w|-i}}{(|w|- i)!}}{\sum_{i=0}^{k} F(\varepsilon,w,i)\cdot \frac{n_m^{|w|-i}}{(|w|- i)!}}=1.
\end{align*}
\end{proof}

We now turn to the case of an infinite sum $\sum c_i.$ In what follows, put $r(v,w):=\lim_{m\to\infty}r(v,w_m).$
\begin{ENLemma}\label{ENlem1}
In the notation introduced above, let $\tau\in\left( {}_K\overline{\mathbb{YF}}\right)^\infty$ be a sequence such that $\sum_{i=0}^k c_i=\infty.$ Then, for every $(v,\ell) \in {}_K\overline{\mathbb{YF}}$, we have
\begin{align}
  \mu_{\tau}(v,\ell)=&D(\varepsilon,v,\ell)\cdot p^\ell(1-p)^{|v|}\cdot\lim_{m\to\infty} \frac{ \sum_{i=0}^{r(v,w)}F(v,w_m,i)\cdot(1-p)^i \cdot p^{r(v,w)-i}}{\sum_{i=0}^{k}F(\varepsilon,w_m,i)\cdot(1-p)^i\cdot p^{k-i}}=\label{ENl21}\\
  =&D(\varepsilon,v,\ell)\cdot p^\ell(1-p)^{|v|}\cdot \lim_{m\to\infty}\frac{ Q_{v, w_m}(p)}{ Q_{\varepsilon,w_m}(p)}.\label{ENl22}
\end{align}
\end{ENLemma}
\begin{proof} By Theorem \ref{ENdf},
\begin{gather}
    \lim_{m \to \infty} \frac{D(v,w_m,n_m-\ell)}{D(\varepsilon,w_m,n_m)}=\lim_{m \to \infty} \frac{\sum_{i=0}^{r(v,w)}F(v,w_m,i)\binom{s(v,w_m)+n_m-\ell-1}{|w_m|-|v|- i }}{ \sum_{i=0}^{k}F(\varepsilon,w_m,i)\binom{\#w_m+n_m-1}{|w_m|- i }}.\label{ENe4}
\end{gather}

   Next, applying \eqref{ENsim1} several times, we see that, for $i\in \overline{r(v,w)}$ and sufficiently large $m$,
    \begin{align}
        \binom{s(v,w_m)+n_m-\ell-1}{|w_m|-|v|- i }=&\frac{(s(v,w_m)+n_m-\ell-1)!}{(|w_m|-|v|- i)!\cdot(-r(v,w)+n_m-\ell-1+i)! }\notag\\
        \sim&\frac{(|w_m|+n_m-1)!}{(|w_m|)!\cdot(n_m-1)! }\cdot \frac{\#w_m^{|v|+i}\cdot n_m^{r(v,w)+\ell-i}}{(\#w_m+n_m)^{r(v,w)+|v|+\ell}}\notag\\
        \sim&\frac{(|w_m|+n_m-1)!}{(|w_m|)!\cdot(n_m-1)! }\cdot (1-p)^{|v|+i}\cdot p^{r(v,w)+\ell-i}.\label{ENchis}
    \end{align}
    
    Similarly, for $i\in \overline{k}$ and sufficiently large $m$,
    \begin{gather}
        \binom{\#w_m+n_m-1}{|w_m|- i }\sim\frac{(|w_m|+n_m-1)!}{(|w_m|)!\cdot(n_m-1)! }\cdot (1-p)^i\cdot p^{k-i}.\label{ENznam}
    \end{gather}
\eqref{ENl21} follows from \eqref{ENe4}, \eqref{ENchis} and \eqref{ENznam}, and \eqref{ENl22} follows from Lemma \ref{ENqf}. 
\end{proof}

Observe that if, under the assumptions of Lemma \ref{ENlem1}, $p(\tau)=0,$ then this is a degenerate measure ($\mu_\tau(v,\ell)=\mathbbm{1}_{\ell=0}$), and if $p(\tau)=1,$ then this is the measure from Lemma \ref{ENt1}. Denote these measures by $\mu_0$ and $\mu_1$, respectively. We shall consider the case $p(\tau)\in(0,1).$

Next, let $f\in\overline{k}$ be a number such that $\sum_{i=0}^{f-1} c_i<\infty,$ $c_f=\infty.$ Also, let
\begin{gather*}
  \overline{w}=\overline{w}(\tau) = 21^{c_{f-1}}2\ldots 21^{c_1}21^{c_0}
\end{gather*}
(if $f=0,$ then $\overline{w}=\varepsilon$).

From the notation for the polynomials $Q$, it is easy to see that, for a given vertex $v\in{}_K\mathbb{YF},$ if there exists $e\in\mathbb{N}_0,$ such that $v\le 1^e\overline{w},$ then
\begin{align*}
     Q_{v, w_m}(p) \sim \left(\prod_{j=0}^{k-f-1}\sum_{i=f}^{f+j}{\beta_i(m)}\right)\cdot  p^{k-f}\cdot \widetilde{Q}_{v,\overline{w}}(p),
\end{align*}
which implies that
\begin{align}
\frac{Q_{v, w_m}(p)}{Q_{\varepsilon, w_m}(p)} \xrightarrow[m\to\infty]{} \frac{\widetilde{Q}_{v,\overline{w}}(p)}{Q_{\varepsilon, \overline{w}}(p)}.\label{ENqsim}
\end{align}

For a given vertex $W\in{}_K\mathbb{YF},$ not beginning with a one, and a given $P\in(0,1)$, introduce the following notation. For each vertex $(v,\ell) \in{}_K \overline{\mathbb{YF}}$, put
\begin{gather*}
    \mu_{P,W}(v,\ell):=D(\varepsilon,v,\ell)\cdot P^\ell(1-P)^{|v|}\frac{ \widetilde{Q}_{v, W}(P) }{ Q_{\varepsilon,W}(P)}.
\end{gather*}

Lemma \ref{ENlem1} and \eqref{ENqsim} imply
\begin{ENZam}\label{ENzamech}
In the notation introduced above, let $\tau\in\left( {}_K\overline{\mathbb{YF}}\right)^\infty$ be a sequence such that $\sum_{i=0}^k c_i=\infty,$ $p(\tau)=P\in(0,1),$ $\overline{w}(\tau)=W.$ Then, for every vertex $(v,\ell) \in {}_K\overline{\mathbb{YF}}$ such that there exists $e\in\mathbb{N}_0,$ such that $v\le 1^eW,$ we have
\begin{gather*}
    \mu_{\tau}(v,\ell)=\mu_{P,W}(v,\ell).
\end{gather*}
\end{ENZam}
Thus, we are ready to state and prove the main theorem.
\begin{ENtheorem}
The Martin boundary of the graph ${}_K\overline{\mathbb{YF}}$ consists of the measures $\mu_0,$ $\mu_1$ and measures of the form $\mu_{P,W}$ for all $W\in{}_K{\mathbb{YF}},$ not beginning with a one, and $P\in(0,1).$ Moreover, all these measures are distinct, and all, except for the measure $\mu_0,$ are probability measures.
\end{ENtheorem}
\begin{proof}
First, observe that, for different $W$, the functions $\mu_{P,W}(v,\ell)$ take nonzero values on different subsets of ${}_K\overline{\mathbb{YF}},$ and, for a fixed $W$, observe that $\mu_{P,W}(\varepsilon,1)=P,$ which implies that all these functions are distinct.

Next, fix $P,W$ and consider a sequence $\tau,$ such that $p(\tau)= P,$ and, for every $m\in\mathbb{N}$, we also have $w_m=1^{\beta(m)}W$ and
\begin{align*}
      \frac{|w_m|}{|w_m|+n_m}\le 1- P.
\end{align*}
It is obvious that such a sequence exists and that $\overline{w}(\tau)=W.$ 

By Remark \ref{ENzamech}, for all $(v,\ell)\in{}_K\overline{\mathbb{YF}}$, we have $\mu_\tau(v,\ell)=\mu_{P,W}(v,\ell),$ and, to complete the proof of the theorem, it suffices to prove that, for every positive integer $\ell$, we have
\begin{gather*}
    \sum_{v\in{}_K\mathbb{YF}} \mu_{\tau}(v,\ell)=\sum_{v\in{}_K\mathbb{YF}} \mu_{P,W}(v,\ell)=1.  
\end{gather*}

Fix $\ell\in\mathbb{N}.$ To apply Lebesgue's dominated convergence theorem, it suffices to prove that there exists a constant $C,$ such that, for every vertex $v\in{}_K{\mathbb{YF}}$ and every $m\in\mathbb{N}$, we have
 \begin{align}
    D(\varepsilon,v,\ell)\cdot \frac{D(v,w_m,n_m-\ell)}{D(\varepsilon,w_m,n_m)}\le C \cdot \mu_{P,W}(v,\ell). \label{ENmajor}
\end{align}

It is clear that the vertices $v,$ satisfying $v\le W,$ are finite in number. Therefore, it suffices to prove the existence of a constant $C$ for every vertex of the form $v=1^eW,$ $e\in\mathbb{N}$ and every $m\in\mathbb{N}$ such that $v\le w_m,$ $\ell\le n_m.$ Under these conditions, by Theorem \ref{ENdf},
\begin{gather*}
     \frac{D(v,w_m,n_m-\ell)}{D(\varepsilon,w_m,n_m)}=\frac{\binom{|w_m|-|v|+n_m-\ell-1}{|w_m|-|v|}}{\sum_{i=0}^{d(W)}F(\varepsilon,W,i)\binom{\#w_m+n_m-1}{|w_m|- i}}.\notag
\end{gather*}

By \eqref{ENznam} and Proposition \ref{ENz1}, there exists a constant $C_1>0$ such that the denominator of this fraction is bounded below by 
\begin{gather*}
    C_1\frac{(|w_m|+n_m)!}{(|w_m|)!\cdot(n_m)!}.
\end{gather*}

Now let us bound the numerator from above:
    \begin{align*}
        \binom{|w_m|-|v|+n_m-\ell-1}{|w_m|-|v|}\le & \binom{|w_m|-|v|+n_m}{|w_m|-|v|}\\
        =&\frac{(|w_m|+n_m)!}{(|w_m|)!\cdot(n_m)! }\cdot \frac{(|w_m|-|v|+1)\cdots (|w_m|)}{(|w_m|+n_m-|v|+1)\cdots (|w_m|+n_m)}\\
        \le&\frac{(|w_m|+n_m)!}{(|w_m|)!\cdot(n_m)! }\cdot \frac{(|w_m|)^{|v|}}{(|w_m|+n_m)^{|v|}}\\
        \le&\frac{(|w_m|+n_m)!}{(|w_m|)!\cdot(n_m)! }\cdot (1-P)^{|v|}.
    \end{align*}

Thus, the left-hand side of \eqref{ENmajor} does not exceed 
\begin{align*}
    D(\varepsilon,v,\ell)\cdot\frac{(1-P)^{|v|}}{C_1}=C\cdot D(\varepsilon,v,\ell)\cdot P^\ell(1-P)^{|v|}\frac{1}{ Q_{\varepsilon,W}(P)}=C\cdot\mu_{P,W}(v,\ell)
\end{align*}
for $C=Q_{\varepsilon,W}(P)/(C_1\cdot P^\ell).$

The theorem is proved.
\end{proof}

\section{Acknowledgments}
I am very grateful to my academic adviser, Fedor Vladimirovich Petrov, for proposing the problem, for his patience in checking the proof and editing the paper, and for his moral support throughout the entire period of work. 

This work was supported by Saint Petersburg State University, project No. 159547043.

\section{AI disclosure}

The translation from Russian original was assisted by 
ChatGPT-6.1 Sol. Mathematical content is due to the author.


\addcontentsline{toc}{section}{Список литературы}

\begin{thebibliography}{}

\bibitem{BE} 
\href{https://www.mathnet.ru/php/archive.phtml?wshow=paper&jrnid=faa&paperid=4139&option_lang=rus}{И.~Бочков, В.~Евтушевский, \emph{Эргодичность границы Мартина графа Юнга --- Фибоначчи}, Функц. анализ и его прил., \textbf{60}:3 (2026), 3–25;} \href{https://link.springer.com/article/10.1134/S1234567826030018}{англ. пер.: I.~Bochkov, V.~Evtushevsky, \emph{Ergodicity of the Martin Boundary of the Young--Fibonacci Graph}, Funct. Anal. Appl., \textbf{60}:3 (2026), 209–227.}

\bibitem{Evtuh1}
 В.~Ю.~Евтушевский, \emph{Перечисление путей в графе Юнга --- Фибоначчи}, \href{https://arxiv.org/abs/2012.06379}{arXiv:2012.06379.}

\bibitem{Evtuh2}
\href{https://www.mathnet.ru/php/archive.phtml?wshow=paper&jrnid=znsl&paperid=6777&option_lang=rus}{В.~Ю.~Евтушевский, \emph{Перечисление путей в графе Юнга --- Фибоначчи}, в кн.: \emph{Теория представлений, динамические системы, комбинаторные и алгоритмические методы.} XXX, Зап. науч. сем. ПОМИ, \textbf{481}, ПОМИ, СПб., 2019, 39-62;} \href{https://link.springer.com/article/10.1007/s10958-020-04829-7}{англ. пер.: V.~Yu.~Evtushevsky, \emph{Enumeration of paths in the Young--Fibonacci graph}, J. Math. Sci. (N.Y.), \textbf{247}:5 (2020), 663–679.}

\bibitem{Fo}
\href{https://www.mathnet.ru/php/archive.phtml?wshow=paper&jrnid=znsl&paperid=5166&option_lang=rus}{С.~В.~Фомин, \emph{Обобщенное соответствие Робинсона --- Шенстеда --- Кнута,} в кн.: \emph{Дифференциальная геометрия, группы Ли и механика.} VIII, Зап. науч. сем. ЛОМИ, \textbf{155}, Изд-во “Наука”, Ленинград. отд., Л., 1986, 156-175;} \href{https://link.springer.com/article/10.1007/BF01247093}{англ. пер.: S.~V.~Fomin, \emph{On generalized Robinson--Schensted--Knuth correspondence,} J. Soviet Math., \textbf{41}:2 (1988), 979-991.}

\bibitem{Bur}
\href{https://hdl.handle.net/11299/142992}{P.~Byrnes, \emph{Structural aspects of differential posets}, Ph.D. thesis, University of Minnesota, Minneapolis, 2012, 80 p.}

\bibitem{Fo1}
\href{https://link.springer.com/article/10.1023/A:1022412010826}{S.~Fomin, \emph{Duality of graded graphs}, J. Algebraic Combin., \textbf{3}:4 (1994), 357–404.}

\bibitem{GK}
\href{https://link.springer.com/article/10.1023/A:1008739619211}{F.~M.~Goodman, S.~V.~Kerov, \emph{The Martin boundary of the Young--Fibonacci lattice}, J. Algebraic Combin., \textbf{11}:1 (2000), 17–48.}


\bibitem{Okada}
\href{https://pubs.ams.org/journals/tran/1994-346-02/S0002-9947-1994-1273538-7}{S.~Okada, \emph{Algebras associated to the Young--Fibonacci lattice}, Trans. Amer. Math. Soc., \textbf{346}:2 (1994), 549–568.}


\bibitem{St}
\href{https://pubs.ams.org/journals/jams/1988-01-04/S0894-0347-1988-0941434-9}{R.~P.~Stanley, \emph{Differential posets}, J. Amer. Math. Soc., \textbf{1}:4 (1988), 919-961.}

\bibitem{Ver}
\href{https://link.springer.com/article/10.1007/s11537-016-1527-z}{A.~M.~Vershik, \emph{Asymptotic theory of path spaces of graded graphs and its applications}, Jpn. J. Math., \textbf{11}:2 (2016), 151–218.}


\end{thebibliography}

\addcontentsline{toc}{section}{References}

\begin{thebibliography}{}

\bibitem{ENBE} 
\href{https://link.springer.com/article/10.1134/S1234567826030018}{I.~Bochkov, V.~Evtushevsky, \emph{Ergodicity of the Martin Boundary of the Young--Fibonacci Graph}, Funct. Anal. Appl., \textbf{60}:3 (2026), 209–227.}

\bibitem{ENEvtuh1}
 V.~Yu.~Evtushevsky, \emph{Enumeration of paths in the Young--Fibonacci graph}, \href{https://arxiv.org/abs/2012.06379}{arXiv:2012.06379.}

\bibitem{ENEvtuh2}
\href{https://link.springer.com/article/10.1007/s10958-020-04829-7}{V.~Yu.~Evtushevsky, \emph{Enumeration of paths in the Young--Fibonacci graph}, J. Math. Sci. (N.Y.), \textbf{247}:5 (2020), 663–679.}

\bibitem{ENFo}
\href{https://link.springer.com/article/10.1007/BF01247093}{S.~V.~Fomin, \emph{On generalized Robinson--Schensted--Knuth correspondence,} J. Soviet Math., \textbf{41}:2 (1988), 979-991.}

\bibitem{ENBur}
\href{https://hdl.handle.net/11299/142992}{P.~Byrnes, \emph{Structural aspects of differential posets}, Ph.D. thesis, University of Minnesota, Minneapolis, 2012, 80 p.}

\bibitem{ENFo1}
\href{https://link.springer.com/article/10.1023/A:1022412010826}{S.~Fomin, \emph{Duality of graded graphs}, J. Algebraic Combin., \textbf{3}:4 (1994), 357–404.}

\bibitem{ENGK}
\href{https://link.springer.com/article/10.1023/A:1008739619211}{F.~M.~Goodman, S.~V.~Kerov, \emph{The Martin boundary of the Young--Fibonacci lattice}, J. Algebraic Combin., \textbf{11}:1 (2000), 17–48.}


\bibitem{ENOkada}
\href{https://pubs.ams.org/journals/tran/1994-346-02/S0002-9947-1994-1273538-7}{S.~Okada, \emph{Algebras associated to the Young--Fibonacci lattice}, Trans. Amer. Math. Soc., \textbf{346}:2 (1994), 549–568.}


\bibitem{ENSt}
\href{https://pubs.ams.org/journals/jams/1988-01-04/S0894-0347-1988-0941434-9}{R.~P.~Stanley, \emph{Differential posets}, J. Amer. Math. Soc., \textbf{1}:4 (1988), 919-961.}

\bibitem{ENVer}
\href{https://link.springer.com/article/10.1007/s11537-016-1527-z}{A.~M.~Vershik, \emph{Asymptotic theory of path spaces of graded graphs and its applications}, Jpn. J. Math., \textbf{11}:2 (2016), 151–218.}


\end{thebibliography}
\end{document}